\documentclass[11pt,reqno]{amsart}

\usepackage{mathrsfs}
\usepackage{amssymb}
\usepackage{amsthm}
\usepackage{amsmath,amsfonts,amssymb,esint}
\usepackage{graphics,color}
\usepackage{enumerate}

\usepackage{marginnote}
\usepackage{xfrac}
\usepackage{mathtools}
\usepackage{hyperref}

\usepackage{comment}

\newtheorem{theorem}{Theorem}[section]
\newtheorem{lemma}[theorem]{Lemma}
\newtheorem{proposition}[theorem]{Proposition}

\newtheorem{definition}[theorem]{Definition\rm}
\newtheorem{remark}{Remark}

\newcommand{\T}{\ensuremath{\mathbb{T}}}
\newcommand*{\R}{\ensuremath{\mathbb{R}}}

\newcommand*{\N}{\ensuremath{\mathbb{N}}}
\newcommand*{\Z}{\ensuremath{\mathbb{Z}}}
\newcommand*{\C}{\ensuremath{\mathbb{C}}}

\renewcommand*{\div}{\ensuremath{\mathrm{div\,}}}
\newcommand*{\tr}{\ensuremath{\mathrm{tr\,}}}

\newcommand*{\Id}{\ensuremath{\mathrm{Id}}}

\newcommand{\eps}{\varepsilon}

\newcommand*{\curl}{\ensuremath{\mathrm{curl\,}}}

\renewcommand*{\P}{\ensuremath{\mathcal{P}}}

\newcommand*{\RR}{\ensuremath{\mathcal{R}}}
\newcommand{\norm}[1]{\left\|#1\right\|}
\newcommand{\abs}[1]{\left|#1\right|}

\newcommand{\vertiii}[1]{{\left\vert\kern-0.25ex\left\vert\kern-0.25ex\left\vert #1 
    \right\vert\kern-0.25ex\right\vert\kern-0.25ex\right\vert}}

\begin{document}

\begin{abstract}
Recently the second and third author developed an iterative scheme for obtaining 
rough solutions of the 3D incompressible Euler equations in H\"older spaces 
({\it arXiv:1202.1751} and {\it arXiv:1205.3626} (2012)). 
The motivation comes from Onsager's conjecture. The construction involves a superposition
of weakly interacting perturbed Beltrami flows on infinitely many scales. An obstruction 
to better regularity arises from the errors in the linear transport of a fast periodic flow by a slow velocity field.

In a recent paper P.~Isett ({\it arXiv:1211.4065}) has improved upon our methods, introducing
some novel ideas on how to deal with this obstruction, thereby reaching a better H\"older exponent -- albeit below the one conjectured by Onsager.
In this paper we give a shorter proof of Isett's final result, adhering more to the original scheme and
introducing some new devices. More precisely we show that for any positive $\eps$ 
there exist periodic solutions of the 3D incompressible Euler
equations which dissipate the total kinetic energy and belong to the H\"older class $C^{\sfrac{1}{5}-\eps}$.
\end{abstract}

\title[Dissipative Euler Flows]
{Transporting microstructure and dissipative Euler flows}

\author{Tristan Buckmaster}
\address{Institut f\"ur Mathematik, Universit\"at Leipzig, D-04103 Leipzig}
\email{tristan.buckmaster@math.uni-leipzig.de}

\author{Camillo De Lellis}
\address{Institut f\"ur Mathematik, Universit\"at Z\"urich, CH-8057 Z\"urich}
\email{camillo.delellis@math.unizh.ch}

\author{L\'aszl\'o Sz\'ekelyhidi Jr.}
\address{Institut f\"ur Mathematik, Universit\"at Leipzig, D-04103 Leipzig}
\email{laszlo.szekelyhidi@math.uni-leipzig.de}

\maketitle

\section{Introduction}

In what follows $\T^3$ denotes the $3$-dimensional torus, i.e. $\T^3 = {\mathbb S}^1\times {\mathbb S}^1 \times
{\mathbb S}^1$. In this note we give a proof of the following theorem.

\begin{theorem}\label{t:main}
Assume $e: [0, 1]\to \R$ is a positive smooth function and $\eps$ a positive number.
Then there is a continuous vector field $v\in C^{\sfrac{1}{5}-\eps} ( \T^3 \times [0, 1], \R^3)$ and a continuous
scalar field $p\in C^{\sfrac{2}{5}-2\eps} (\T^3\times [0,1])$ which solve the incompressible
Euler equations 
\begin{equation}\label{e:Euler}
\left\{\begin{array}{l}
\partial_t v + \div (v\otimes v) + \nabla p =0\\ \\
\div v = 0
\end{array}\right.
\end{equation}
in the sense of distributions and such that
\begin{equation}\label{e:energy_id}
e(t) = \int |v|^2 (x,t)\, dx\qquad\forall t\in [0, 1]\, .
\end{equation}
\end{theorem}

Results of this type are associated with the famous conjecture of Onsager. In a nutshell, the question is about whether or not weak solutions in a given regularity class satisfy the law of energy conservation or not. For classical solutions (say, $v\in C^1$) we can multiply \eqref{e:Euler} by $v$ itself, integrate by parts and obtain the energy balance
\begin{equation}\label{e:energycons}
\int_{\T^3}|v(x,t)|^2\,dx=\int_{\T^3}|v(x,0)|^2\,dx\quad\textrm{ for all }t>0.
\end{equation}
On the other hand, for weak solutions (say, merely $v\in L^2$) 
\eqref{e:energycons} might be violated, and this possibility has been considered for a rather long time in the context of $3$ dimensional turbulence. 
In his famous note \cite{Onsager} about statistical hydrodynamics, Onsager considered 
weak solutions satisfying the H\"older condition
\begin{equation}\label{e:hoelderest}
|v(x,t)-v(x',t)|\leq C|x-x'|^\theta,
\end{equation}
where the constant $C$ is independent of $x,x'\in\T^3$ and $t$. He conjectured that 
\begin{enumerate}
\item[(a)] Any weak solution $v$ satisfying \eqref{e:hoelderest} with $\theta>\frac{1}{3}$ conserves the energy;
\item[(b)] For any $\theta<\frac{1}{3}$ there exist weak solutions $v$ satisfying \eqref{e:hoelderest} which do not conserve the energy.
\end{enumerate}
This conjecture is also very closely related to Kolmogorov's famous K41 theory \cite{Kolmogorov} for homogeneous isotropic turbulence in $3$ dimensions. We refer the interested 
reader to \cite{FrischBook,Robert,EyinkSreenivasan}.
Part (a) of the conjecture is by now fully resolved: it has first been considered by Eyink in \cite{Eyink} following Onsager's original calculations and proved by Constantin, E and Titi 
in \cite{ConstantinETiti}. Slightly weaker assumptions on $v$ (in Besov spaces) were subsequently shown to be sufficient for energy conservation 
in \cite{RobertDuchon,CCFS2007}.

In this paper we are concerned with part (b) of the conjecture. Weak solutions violating the energy equality have been constructed for a long time, starting with the seminal work of Scheffer 
and Shnirelman \cite{Scheffer93,Shnirelmandecrease}.
In \cite{DS1,DS2} a new point of view was introduced, relating 
the issue of energy conservation to Gromov's h-principle, see also \cite{DSsurvey}. 
In \cite{DS3} and \cite{DS4} the first constructions of continuous and H\"older-continuous 
weak solutions violating the energy equality appeared. In particular in \cite{DS4} the authors proved 
Theorem \ref{t:main} with H\"older exponent $1/10-\eps$ replacing $1/5-\eps$. 

The threshold exponent $\frac{1}{5}$ has been recently reached by P.~Isett in \cite{Isett} (although strictly speaking he proves
a variant of Theorem \ref{t:main}, since he shows the existence of nontrivial solutions which are compactly supported in time,
rather than prescribing the total kinetic energy).
Our aim in this note is to give a shorter proof of Isett's improvement in the H\"older exponent and isolate the main new ideas
of \cite{Isett} compared to \cite{DS3,DS4}. We observe in passing that the arguments given here can 
be easily modified to produce nontrivial solutions with compact support in time, but losing control on the exact shape
of the energy. The question of producing a solution matching an energy profile $e$ which might vanish is subtler.
A similar issue has been recently treated in the paper \cite{Daneri}.

\subsection{Euler-Reynolds system and the convex integration scheme}
Let us recall the main ideas, on which the constructions in \cite{DS3,DS4} are based. 

The proof is achieved through an iteration scheme. At each step $q\in \N$ we construct a triple $(v_q, p_q, \mathring{R}_q)$ solving the Euler-Reynolds system 
(see \cite[Definition 2.1]{DS3}):
\begin{equation}\label{e:euler_reynolds}
\left\{\begin{array}{l}
\partial_t v_q + \div (v_q\otimes v_q) + \nabla p_q =\div\mathring{R}_q\\ \\
\div v_q = 0\, .
\end{array}\right.
\end{equation} 
The {\it size} of the perturbation 
$$
w_q:=v_{q}-v_{q-1}
$$
will be measured by two parameters: $\delta_q^{\sfrac12}$ is the {\it amplitude} and $\lambda_q$ the {\it frequency}. More precisely, denoting the (spatial) H\"older norms by $\|\cdot\|_k$ (see Section \ref{s:hoelder} for precise definitions),
\begin{align}
\|w_{q}\|_0 &\leq M \delta_{q}^{\sfrac{1}{2}}\,,\label{e:v_C0_iter}\\
\|w_{q}\|_1 &\leq M \delta_{q}^{\sfrac{1}{2}} \lambda_{q}\,,\label{e:v_C1_iter}
\end{align}
and similarly,
\begin{align}
\|p_{q}-p_{q-1}\|_0 &\leq M^2 \delta_{q}\,, \label{e:p_C0_iter}\\
\|p_{q}-p_{q-1}\|_1 &\leq M^2 \delta_{q} \lambda_{q}\,,\label{e:p_C1_iter}
\end{align}
where $M$ is a constant depending only on the function $e=e(t)$ in the Theorem. 
 
In constructing the iteration, the new perturbation, $w_{q+1}$ will be chosen so as to balance the 
previous Reynolds error $\mathring{R}_q$, in the sense that (cf.\ equation \eqref{e:euler_reynolds}) 
we have $\|w_{q+1}\otimes w_{q+1}\|_0\sim \|\mathring{R}_q\|_0$. This is formalized as
\begin{align}
\|\mathring{R}_q\|_0 &\leq \eta \delta_{q+1}\,,\label{e:R_C0_iter}\\
\|\mathring{R}_q\|_1&\leq M \delta_{q+1}\lambda_q\,, \label{e:R_C1_iter}
\end{align}
where $\eta$ will be a small constant, again only depending on $e=e(t)$ in the Theorem. 
Estimates of type \eqref{e:v_C0_iter}-\eqref{e:R_C1_iter} appear already in the paper
\cite{DS4}: although the bound claimed for $\|\mathring{R}_1\|_1$ in the main proposition
of \cite{DS4} is the weaker one $\|\mathring{R}_q\|_1\leq M \delta_q^{\sfrac{1}{2}}\lambda_q$
(cf.\  \cite[Proposition 2.2]{DS4}): $\lambda_q$ here
corresponds to $(D\delta/\bar{\delta}^2)^{1+\eps}$ there), this was just done for the
ease of notation and the actual bound achieved in the proof does in fact correspond
to \eqref{e:R_C1_iter} (cf.\ Step 4 in Section 9). In the language of \cite{Isett} 
the estimates \eqref{e:v_C0_iter}-\eqref{e:R_C1_iter} correspond to the frequency energy
levels of order 0 and 1 (cf.\  Definition 9.1 therein). 

Along the iteration we will have 
$$
\delta_q\to 0\quad\textrm{ and }\quad \lambda_q\to \infty
$$
at a rate that is {\it at least} exponential. On the one hand \eqref{e:v_C0_iter}, \eqref{e:p_C0_iter} and \eqref{e:R_C0_iter} will imply the convergence of the sequence $v_q$ to a continuous weak solution of the Euler equations. On the other hand the precise dependence of $\lambda_q$ on $\delta_q$ will determine the critical H\"older regularity. Finally, the equation \eqref{e:energy_id} will be ensured by 
\begin{equation}\label{e:energy_iter}
\left| e(t) (1-\delta_{q+1}) - \int |v_q|^2 (x,t)\, dx \right| \leq \frac{1}{4} \delta_{q+1} e(t)\,.
\end{equation}
Note that, being an expression quadratic in $v_q$, this estimate is consistent with \eqref{e:R_C0_iter}.

As for the perturbation, it will consist essentially of a finite sum of modulated Beltrami modes (see Section \ref{s:prelim} below), so that 
$$
w_q(x,t)=\sum_{k}a_k(x,t)\,\phi_k(x,t)\,B_ke^{i\lambda_{q}k\cdot x}\,,
$$
where $a_k$ is the amplitude, $\phi_k$ is a phase function (i.e. $|\phi_k|=1$) and $B_ke^{i\lambda_{q}k\cdot\ x}$ is a complex Beltrami mode at frequency $\lambda_{q}$. Having a perturbation of this form ensures that the ``oscillation part of the error''
$$
\div(w_{q}\otimes w_{q}+\mathring{R}_{q-1})
$$
in the equation \eqref{e:euler_reynolds} vanishes, see \cite{DS3} (Isett in \cite{Isett} calls this term ``high-high interaction''). The main analytical part of the argument goes in to choosing $a_k$ and $\phi_k$ correctly in order to deal with the so-called transport part of the error 
$$
\partial_t w_{q}+v_{q-1}\cdot \nabla w_{q}\,.
$$
In \cite{DS3,DS4} a second large parameter $\mu(=\mu_{q})$ was introduced to deal with this term. In some sense the role of $\mu$ is to interpolate between errors of order $1$ in the transport term and errors of order $\lambda_{q}^{-1}$ in the oscillation term. 

The technique used in \cite{DS4} for the transport term leads to the H\"older exponent $\frac{1}{10}$. 
In our opinion the key new idea introduced by Isett is to recognize that the transport error can be reduced 
by defining $a_k$ and $\phi_k$ in such a way that adheres more closely to the transport structure of the equation. 
This requires two new ingredients.  First, the phase functions $\phi_k$ are defined using the flow map of 
the vector field $v_q$, whereas in \cite{DS4} they were functions of $v_q$ itself. With the latter choice, 
although some improvement of the exponent $1/10$ is possible, the threshold $1/5$ seems beyond reach. 
Secondly, Isett introduces a new set of estimates to complement 
\eqref{e:v_C0_iter}-\eqref{e:energy_iter} with the purpose of controlling the free transport evolution of the Reynolds error:
\begin{equation}
\|\partial_t \mathring{R}_q + v_q \cdot \nabla \mathring{R}_q\|_0 \leq \delta_{q+1} \delta_q^{\sfrac{1}{2}} \lambda_q\, .\label{e:R_Dt_iter}
\end{equation}
These two ingredients play a key role also in the proof of Theorem \ref{t:main} given here; 
however, compared to \cite{Isett}, we improve upon the simplicity of their implementation.
In order to compare our proof to Isett's proof, it is worth to notice that the parameter $\mu$ corresponds to
the inverse of the life-span parameter $\tau$ used in \cite{Isett}.

\subsection{Improvements}
Although the construction of Isett in \cite{Isett} is essentially based on this same scheme outlined here, there are a number of further points of departure. For instance, Isett considers perturbations with a nonlinear phase rather than the simple 
stationary flows used here, and consequently, he uses a ``microlocal" version of the Beltrami flows. This also leads to 
the necessity of appealing to nonlinear stationary phase lemmas. Our purpose here is to show that, 
although the other ideas exploited in \cite{Isett} are of independent interest and might also, in principle,
lead to better bounds in the future, with the additional control in \eqref{e:R_Dt_iter}, 
a scheme much more similar to the one introduced in \cite{DS3} provides a substantially 
shorter proof of Theorem \ref{t:main}. To this end, however, we introduce some new 
devices which greatly simplifies the relevant estimates:
\begin{itemize}
\item[(a)] We regularize the maps $v_q$ and $\mathring{R}_q$  in space only and then
solve locally in time the free-transport equation in order to approximate $\mathring{R}_q$. 
\item[(b)] Our maps $a_k$ are then elementary algebraic functions of the approximation of $\mathring R_q$.
\item[(c)] The estimates for the Reynolds stress are still carried on based on simple stationary ``linear'' phase arguments.
\item[(d)] The proof of \eqref{e:R_Dt_iter} is simplified by one commutator estimate which, 
in spite of having a classical flavor, deals efficiently with one important error term.
\end{itemize}

\subsection{The main iteration proposition and the proof of Theorem \ref{t:main}}

Having outlined the general idea above, we proceed with the iteration, starting with
 the trivial solution $(v_0,p_0,\mathring{R}_0)=(0,0,0)$. We will construct new triples $(v_q,p_q,\mathring{R}_q)$ inductively, assuming the estimates \eqref{e:v_C0_iter}-\eqref{e:R_Dt_iter}.

\begin{proposition}\label{p:iterate}
There are positive constants $M$ and $\eta$ depending only on $e$ such that the following holds. 
For every $c>\frac{5}{2}$ and $b>1$, if $a$ is sufficiently large, then there is a sequence of triples 
$(v_q, p_q, \mathring{R}_q)$ starting with $(v_0, p_0, \mathring{R}_0)= (0,0,0)$, 
solving \eqref{e:euler_reynolds} and satisfying the estimates \eqref{e:v_C0_iter}-\eqref{e:R_Dt_iter}, where 
$\delta_q:= a^{-b^q}$, $\lambda_q \in [a^{cb^{q+1}}, 2 a^{cb^{q+1}}]$ for $q=0,1,2,\dots$. In addition we claim the estimates
\begin{align}
\|\partial_t (v_q-v_{q-1})\|_0 \leq C \delta_q^{\sfrac{1}{2}}\lambda_q \qquad\mbox{and}\qquad
\|\partial_t (p_q-p_{q-1})\|_0\leq C \delta_q \lambda_q\label{e:t_derivatives}
\end{align}
\end{proposition}

\begin{proof}[Proof of Theorem \ref{t:main}]
Choose any $c>\frac{5}{2}$ and $b>1$ and let $(v_q, p_q,\mathring{R}_q)$ be a sequence as in Proposition \ref{p:iterate}.
It follows then easily that $\{(v_q, p_q)\}$ converge uniformly to a pair of continuous functions $(v,p)$ such that
\eqref{e:Euler} and \eqref{e:energy_id} hold. 
We introduce the notation $\|\cdot\|_{C^\vartheta}$ for H\"older norms in space and time.
From \eqref{e:v_C0_iter}-\eqref{e:p_C1_iter}, \eqref{e:t_derivatives} and interpolation we conclude
\begin{align}
\|v_{q+1}-v_q\|_{C^\vartheta} & \leq M \delta_{q+1}^{\sfrac{1}{2}} \lambda_{q+1}^\vartheta 
\leq C a^{b^{q+1} (2cb \vartheta -1)/2}\\
\|p_{q+1}-p_q\|_{C^{2\vartheta}} & \leq M^2 \delta_{q+1} \lambda_{q+1}^{2\vartheta} \leq 
C a^{b^{q+1} (2cb \vartheta -1)}\, .
\end{align}
Thus, for every $\vartheta< \frac{1}{2bc}$, $v_q$ converges in $C^{\vartheta}$ and $p_q$ in $C^{2\vartheta}$.
\end{proof}

\subsection{Plan of the paper} In the rest of the paper we will use $D$ and $\nabla$ for 
differentiation in the space variables and $\partial_t$ for differentiation in the time variable. 
After recalling in Section \ref{s:prelim} some preliminary notation from the paper \cite{DS3}, in Section \ref{s:perturbations} we give the precise
definition of the maps $(v_{q+1}, p_{q+1}, \mathring{R}_{q+1})$ assuming the triple $(v_q, p_q, \mathring{R}_q)$ to be known. The Sections \ref{s:perturbation_estimates}, \ref{s:estimates_energy} and \ref{s:reynolds} will focus on estimating, respectively, $w_{q+1} =v_{q+1}-v_q$, $\int |v_{q+1}|^2 (x,t)\, dx$ and $\mathring{R}_{q+1}$.
These estimates are then collected in Section \ref{s:conclusion} where Proposition \ref{p:iterate} will be finally proved. The appendix collects several technical (and, for the most part, well-known) estimates on the different classical PDEs involved in our construction, i.e. the transport equation, the Poisson equation and the biLaplace equation.

\subsection{Acknowledgements}
We wish to thank Phil Isett for several very interesting discussions and suggestions for improvements on this manuscript. 

T.B. and L.Sz. acknowledge the support of the ERC Grant Agreement No. 277993, C.dL. acknowledges the support of the 
SNF Grant 129812.

\section{Preliminaries}\label{s:prelim}

\subsection{Geometric preliminaries}
In this paper we denote by $\R^{n\times n}$, as usual, the space of $n\times n$ matrices, whereas $\mathcal{S}^{n\times n}$
and $\mathcal{S}^{n\times n}_0$ denote, respectively, the corresponding subspaces of symmetric matrices 
and of trace-free symmetric matrices. The $3\times 3$ identity matrix will be denoted with $\Id$. 
For definitiveness we will use the matrix operator norm $|R|:=\max_{|v|=1}|Rv|$. Since we will
deal with symmetric matrices, we have the identity $|R|= \max_{|v|=1} |Rv \cdot v|$.

\begin{proposition}[Beltrami flows]\label{p:Beltrami}
Let $\bar\lambda\geq 1$ and let $A_k\in\R^3$ be such that 
$$
A_k\cdot k=0,\,|A_k|=\tfrac{1}{\sqrt{2}},\,A_{-k}=A_k
$$
for $k\in\Z^3$ with $|k|=\bar\lambda$.
Furthermore, let 
$$
B_k=A_k+i\frac{k}{|k|}\times A_k\in\C^3.
$$
For any choice of $a_k\in\C$ with $\overline{a_k} = a_{-k}$ the vector field
\begin{equation}\label{e:Beltrami}
W(\xi)=\sum_{|k|=\bar\lambda}a_kB_ke^{ik\cdot \xi}
\end{equation}
is real-valued, divergence-free and satisfies
\begin{equation}\label{e:Bequation}
\div (W\otimes W)=\nabla\frac{|W|^2}{2}.
\end{equation}
Furthermore
\begin{equation}\label{e:av_of_Bel}
\langle W\otimes W\rangle= \fint_{\T^3} W\otimes W\,d\xi = \frac{1}{2} \sum_{|k|=\bar\lambda} |a_k|^2 \left( \Id - \frac{k}{|k|}\otimes\frac{k}{|k|}\right)\, .  
\end{equation}
\end{proposition}

The proof of \eqref{e:Bequation}, which is quite elementary and can be found in \cite{DS3}, is based on the following algebraic identity, which we state separately for future reference:
\begin{lemma}\label{l:BkBk'}
Let $k,k'\in \Z^3$ with $|k|=|k'|=\bar{\lambda}$ and let $B_k,B_{k'}\in \C^3$ be the associated vectors from Proposition \ref{p:Beltrami}. Then we have
$$
(B_k\otimes B_{k'}+B_{k'}\otimes B_k)(k+k')=(B_k\cdot B_{k'})(k+k').
$$
\end{lemma}
\begin{proof}
The proof is a straight-forward calculation. Indeed, since $B_k\cdot k=B_{k'}\cdot k'=0$, we have
\begin{align*}
(B_k\otimes B_{k'}+&B_{k'}\otimes B_k)(k+k')=(B_{k'}\cdot k)B_k+(B_{k}\cdot k')B_{k'}\\
&=-B_k\times (k'\times B_{k'})-B_{k'}\times (k\times B_{k})+(B_k\cdot B_{k'})(k+k')\\
&=i\bar{\lambda}(B_k\times B_{k'}+B_{k'}\times B_k)+(B_k\cdot B_{k'})(k+k'),
\end{align*}
where the last equality follows from
$$
k\times B_k=-i\bar{\lambda}B_k\textrm{ and }k'\times B_{k'}=-i\bar{\lambda}B_{k'}\,.
$$
\end{proof}

Another important ingredient is the following geometric lemma, also taken from \cite{DS3}.

\begin{lemma}[Geometric Lemma]\label{l:split}
For every $N\in\N$ we can choose $r_0>0$ and $\bar{\lambda} > 1$ with the following property.
There exist pairwise disjoint subsets 
$$
\Lambda_j\subset\{k\in \Z^3:\,|k|=\bar{\lambda}\} \qquad j\in \{1, \ldots, N\}
$$
and smooth positive functions 
\[
\gamma^{(j)}_k\in C^{\infty}\left(B_{r_0} (\Id)\right) \qquad j\in \{1,\dots, N\}, k\in\Lambda_j
\]
such that
\begin{itemize}
\item[(a)] $k\in \Lambda_j$ implies $-k\in \Lambda_j$ and $\gamma^{(j)}_k = \gamma^{(j)}_{-k}$;
\item[(b)] For each $R\in B_{r_0} (\Id)$ we have the identity
\begin{equation}\label{e:split}
R = \frac{1}{2} \sum_{k\in\Lambda_j} \left(\gamma^{(j)}_k(R)\right)^2 \left(\Id - \frac{k}{|k|}\otimes \frac{k}{|k|}\right) 
\qquad \forall R\in B_{r_0}(\Id)\, .
\end{equation}
\end{itemize}
\end{lemma}

\subsection{The operator $\mathcal{R}$} Following \cite{DS3}, we introduce the following operator
in order to deal with the Reynolds stresses.

\begin{definition}\label{d:reyn_op}
Let $v\in C^\infty (\T^3, \R^3)$ be a smooth vector field. 
We then define $\RR v$ to be the matrix-valued periodic function
\begin{equation*}
\RR v:=\frac{1}{4}\left(\nabla\P u+(\nabla\P u)^T\right)+\frac{3}{4}\left(\nabla u+(\nabla u)^T\right)-\frac{1}{2}(\div u) \Id,
\end{equation*}
where $u\in C^{\infty}(\T^3,\R^3)$ is the solution of
\begin{equation*}
\Delta u=v-\fint_{\T^3}v\textrm{ in }\T^3
\end{equation*}
with $\fint_{\T^3} u=0$ and $\P$ is the Leray projection onto divergence-free fields with zero average.
\end{definition}

\begin{lemma}[$\RR=\textrm{div}^{-1}$]\label{l:reyn}
For any $v\in C^\infty (\T^3, \R^3)$ we have
\begin{itemize}
\item[(a)] $\RR v(x)$ is a symmetric trace-free matrix for each $x\in \T^3$;
\item[(b)] $\div \RR v = v-\fint_{\T^3}v$.
\end{itemize}
\end{lemma}

\section{The inductive step}\label{s:perturbations}

In this section we specify the inductive procedure which allows to construct $(v_{q+1}, p_{q+1}, \mathring{R}_{q+1})$ from $(v_q, p_q, \mathring{R}_q)$. Note that the choice of the sequences $\{\delta_q\}_{q\in\N}$ and $\{\lambda_q\}_{q\in\N}$ specified in Proposition \ref{p:iterate} implies that, for a sufficiently large $a>1$, depending only on $b>1$ and $c>5/2$, we have:
\begin{equation}\label{e:summabilities}
\sum_{j\leq q} \delta_j \lambda_j \leq 2 \delta_q \lambda_q\, ,\quad
1 \leq \sum_{j\leq q} \delta_j^{\sfrac{1}{2}} \lambda_j \leq 2\delta_q^{\sfrac{1}{2}} \lambda_q\, ,\quad
 \sum_j \delta_j \leq \sum_j \delta_j^{\sfrac{1}{2}} \leq 2\, .
\end{equation}
Since we are concerned with a single step in the iteration, with a slight abuse of notation we will write $(v,p,\mathring{R})$ for $(v_q, p_q, \mathring{R}_q)$ and
$(v_1,p_1, \mathring{R}_1)$ for $(v_{q+1}, p_{q+1}, \mathring{R}_{q+1})$.
Our inductive hypothesis implies then the following set of estimates:
\begin{align}
&\|v\|_0 \leq 2 M,\qquad \|v\|_1 \leq 2M \delta_{q}^{\sfrac{1}{2}} \lambda_{q}\,,\label{e:v_old}\\
&\|\mathring{R}\|_0 \leq \eta\delta_{q+1}, \,\quad \|\mathring{R}\|_1 \leq M\delta_{q+1}\lambda_q\,, \label{e:R_old}\\
&\|p\|_0 \leq 2 M^2,  \qquad \|p\|_1 \leq 2M^2 \delta_q \lambda_q\,,\label{e:p_old}
\end{align}
and
\begin{equation}
\|(\partial_t+v\cdot\nabla)\mathring{R}\|_0\leq M\delta_{q+1}\delta_q^{\sfrac12}\lambda_q\,.\label{e:DtR_old}
\end{equation}

\bigskip

The new velocity $v_1$ will be defined as a sum 
$$
v_1:= v + w_o + w_c,
$$
where $w_o$ is the principal perturbation and $w_c$ is a corrector.
 The ``principal part'' of the perturbation $w$ will be a sum of Beltrami flows
 \[
w_o (t,x) := \sum_{|k|=\lambda_0} a_{k}\phi_{k}B_ke^{i\lambda_{q+1}k\cdot x}\, ,
\]
where $B_ke^{i\lambda_{q+1}k\cdot x}$ is a single Beltrami mode at frequency $\lambda_{q+1}$, with phase shift $\phi_{k}=\phi_{k}(t,x)$ (i.e. $|\phi_{k}|=1$) and amplitude $a_{k}=a_{k}(t,x)$. In the following subsections we will define $a_k$ and $\phi_k$.

\subsection{Space regularization of $v$ and $R$}

We fix a symmetric non-negative convolution kernel $\psi\in C^\infty_c (\R^3)$ and a small parameter $\ell$ (whose choice will be specified later). Define
$v_\ell:=v*\psi_\ell$ and $\mathring{R}_\ell:=\mathring{R}*\psi_\ell$,
where the convolution is in the $x$ variable only. 
Standard estimates on regularizations by convolution lead to the following:
\begin{align}
&\|v-v_{\ell}\|_0\leq C\,\delta_q^{\sfrac12}\lambda_q\ell,\label{e:v-v_ell}\\
&\|\mathring{R}-\mathring{R}_{\ell}\|_0\leq C\,\delta_{q+1}\ell,\label{e:R-R_ell}
\end{align}
and for any $N\geq 1$ there exists a constant $C=C(N)$ so that
\begin{align}
&\|v_\ell\|_N\leq C\,\delta_q^{\sfrac12}\lambda_q\ell^{1-N},\label{e:v_ell}\\
&\|\mathring{R}_\ell\|_N\leq C\,\delta_{q+1}\lambda_q\ell^{1-N}.\label{e:R_ell}
\end{align}

\subsection{Time discretization and transport for the Reynolds stress}

Next, we fix a smooth cut-off function $\chi\in C^\infty_c ((-\frac{3}{4}, \frac{3}{4})$ such that 
$$
\sum_{l\in \Z} \chi^2 (x-l) = 1,
$$ 
and a large parameter $\mu\in \N\setminus \{0\}$, whose choice will be specified later.

For any $l\in [0, \mu]$ we define
\[
\rho_l:= \frac{1}{3 (2\pi)^3} \left(e (l\mu^{-1}) \left(1-\delta_{q+2}\right) - \int_{\T^3} |v|^2 (x,l\mu^{-1})\, dx\right).
\]
Note that \eqref{e:energy_iter} implies
$$
\frac{1}{3(2\pi)^3}e (l\mu^{-1})(\tfrac{3}{4}\delta_{q+1}-\delta_{q+2})\leq \rho_l\leq \frac{1}{3(2\pi)^3}e (l\mu^{-1})(\tfrac{5}{4}\delta_{q+1}-\delta_{q+2}).
$$
We will henceforth assume
$$
\delta_{q+2}\leq \frac{1}{2}\delta_{q+1},
$$
so that we obtain
\begin{equation}\label{e:est_rhol}
C_0^{-1}(\min e)\delta_{q+1}\leq \rho_l\leq C_0(\max e)\delta_{q+1},
\end{equation}
where $C_0$ is an absolute constant.

Finally, define $R_{\ell,l}$ to be the unique solution to the transport equation
\begin{equation}\label{e:R-transport}
\left\{\begin{array}{l}
\partial_t \mathring{R}_{\ell,l} +v_\ell\cdot\nabla \mathring{R}_{\ell,l} = 0 \\
\mathring{R}_{\ell,l}(x,\frac l{\mu})=\mathring{R}_{\ell}(x,\frac{l}{\mu})\, .
\end{array}\right.
\end{equation}
and set
\begin{equation}\label{e:mathringRnew}
R_{\ell,l}(x,t):=\rho_l\Id-\mathring{R}_{\ell,l}(x,t).
\end{equation}

\subsection{The maps $v_1, w, w_o$ and $w_c$}\label{ss:def_w}

We next consider $v_\ell$ as a $2\pi$-periodic function on $\R^3\times [0,1]$ and, for every $l\in [0, \mu]$, we let $\Phi_l: \R^3\times [0,1]\to \R^3$ be the solution of 
\begin{equation}\label{e:Phi-transport}
\left\{\begin{array}{l}
\partial_t \Phi_l + v_{\ell}\cdot  \nabla \Phi_l =0\\ \\
\Phi_l (x,l \mu^{-1})=x
\end{array}\right.
\end{equation}
Observe that $\Phi_l (\cdot, t)$ is the inverse of the flow of the periodic vector-field $v_\ell$, starting at time $t=l\mu^{-1}$ as the identity. Thus, if $y\in (2\pi \Z)^3$, then $\Phi_l (x, t) - \Phi_l (x+y, t) \in (2\pi \Z)^3$:
$\Phi_l (\cdot, t)$ can hence be thought as a diffeomorphism of $\T^3$ onto itself and, for every $k\in \Z^3$, the map $\T^3 \times [0,1] \ni (x,t) \to e^{i \lambda_{q+1} k \cdot \Phi_l (x,t)}$ is well-defined.

We next apply Lemma \ref{l:split} with $N=2$, denoting by $\Lambda^e$ and $\Lambda^o$ the corresponding families of frequencies in $\Z^3$, and set $\Lambda := \Lambda^o$ + $\Lambda^e$. For each $k\in \Lambda$ and each $l\in \Z\cap[0,\mu]$ we then set
\begin{align}
\chi_l(t)&:=\chi\Bigl(\mu(t-l)\Bigr),\label{e:chi_l}\\
a_{kl}(x,t)&:=\sqrt{\rho_l}\gamma_k \left(\frac{R_{\ell,l}(x,t)}{\rho_l}\right),\label{e:a_kl}\\
w_{kl}(x,t)& := a_{kl}(x,t)\,B_ke^{i\lambda_{q+1}k\cdot \Phi_l(x,t)}.\label{e:w_kl}
\end{align}
The ``principal part'' of the perturbation $w$ consists of the map
\begin{align}\label{e:def_wo}
w_o (x,t) := \sum_{\textrm{$l$ odd}, k\in \Lambda^o} \chi_l(t)w_{kl} (x,t) +
\sum_{\textrm{$l$ even}, k\in \Lambda^e} \chi_l(t)w_{kl} (x,t)\, .
\end{align}
From now on, in order to make our notation simpler, we agree that the pairs of indices 
$(k,l)\in \Lambda\times [0, \mu]$ which enter in our summations satisfy always 
the following condition: $k\in \Lambda^e$ when $l$ is even and $k\in \Lambda^o$ when $l$ is odd.

It will be useful to introduce the ``phase" 
\begin{equation}\label{e:phi_kl}
\phi_{kl}(x,t)=e^{i\lambda_{q+1}k\cdot[\Phi_l(x,t)-x]},
\end{equation}
with which we obviously have
\[
\phi_{kl}\cdot e^{i\lambda_{q+1}k\cdot x}=e^{i\lambda_{q+1}k\cdot\Phi_l}.
\]
Since $R_{\ell,l}$ and $\Phi_l$ are defined as solutions of the transport equations \eqref{e:R-transport} and \eqref{e:Phi-transport}, we have
\begin{align}\label{e:Dt_varia=0}
(\partial_t+v_\ell\cdot\nabla)a_{kl}=0\qquad\textrm{ and }\qquad(\partial_t+v_\ell\cdot\nabla)e^{i\lambda_{q+1}k\cdot \Phi_l(x,t)}=0,
\end{align}
hence also 
\begin{equation}\label{e:Dt_w_kl=0}
(\partial_t+v_\ell\cdot\nabla)w_{kl}=0.
\end{equation}
The corrector $w_c$ is then defined in such a way that $w:= w_o+w_c$ is divergence free:
\begin{align}
w_c&:= \sum_{kl} \frac{\chi_l}{\lambda_{q+1}}\curl\left(ia_{kl}\phi_{kl}\frac{k\times B_k}{|k|^2}\right) e^{i\lambda_{q+1}k \cdot x}\nonumber\\
&=\sum_{kl}\chi_l\Bigl(\frac{i}{\lambda_{q+1}}\nabla a_{kl}-a_{kl}(D\Phi_{l}-\Id)k\Bigr)\times\frac{k\times B_k}{|k|^2}e^{i\lambda_{q+1}k\cdot\Phi_l}\label{e:corrector}
\end{align}

\begin{remark} To see that $w=w_o+w_c$ is divergence-free, just note that, since $k\cdot B_k=0$, we have
$k\times (k\times B_k)=-|k|^2B_k$ and hence $w$ can be written as
\begin{equation}\label{e:w_compact_form}
w = \frac{1}{\lambda_{q+1}} \sum_{(k,l)} \chi_l \,\curl \left( i a_{kl}\,\phi_{kl}\,\frac{k\times B_k}{|k|^2} e^{i\lambda_{q+1}k \cdot x}\right)
\end{equation}
\end{remark}

For future reference it is useful to introduce the notation
\begin{equation}\label{e:L_kl}
L_{kl}:=a_{kl}B_k+\Bigl(\frac{i}{\lambda_{q+1}}\nabla a_{kl}-a_{kl}(D\Phi_{l}-\Id)k\Bigr)\times\frac{k\times B_k}{|k|^2},
\end{equation}
so that the perturbation $w$ can be written as
\begin{equation}\label{e:w_Lform}
w=\sum_{kl}\chi_l\,L_{kl}\,e^{i\lambda_{q+1}k\cdot\Phi_l}\,.
\end{equation}
Moreover, we will frequently deal with the transport derivative with respect to the regularized flow $v_\ell$ of various expressions, and will henceforth use the notation
$$
D_t:=\partial_t+v_{\ell}\cdot\nabla.
$$

\subsection{Determination of the constants $\eta$ and $M$} 
In order to determine $\eta$, first of all recall from Lemma \ref{l:split} that the functions $a_{kl}$ are
well-defined provided 
\begin{equation*}
\biggl |\frac{R_{\ell,l}}{\rho_l} - \Id \biggr | \leq r_0\, ,
\end{equation*}
where $r_0$ is the constant of Lemma \ref{l:split}. Recalling the definition of $R_{\ell,l}$ we easily deduce from the maximum principle for transport equations (cf.\ \eqref{e:max_prin} in Proposition \ref{p:transport_derivatives}) that
$\|\mathring{R}_{\ell,l}\|_0\leq \|\mathring{R}\|_0$.
Hence, from \eqref{e:R_C0_iter} and \eqref{e:est_rhol} we obtain
\[
\biggl |\frac{R_{\ell, l}}{\rho_l} - \Id \biggr | \leq C_0 \frac{\eta}{\min e},
\]
and thus we will require that 
\begin{equation*}
C_0 \frac{\eta}{\min e} \leq \frac{r_0}{4}\, .
\end{equation*}

\medskip

The constant $M$ in turn is determined by comparing the estimate \eqref{e:v_C0_iter} for $q+1$ with the definition of the principal perturbation $w_o$ in \eqref{e:def_wo}. Indeed, using \eqref{e:chi_l}-\eqref{e:def_wo} and \eqref{e:est_rhol} we have
$\|w_o\|_{0}\leq C_0|\Lambda|(\max e)\delta_{q+1}^{\sfrac12}$.
We therefore set 
\begin{equation*}
M=2C_0|\Lambda|(\max e),
\end{equation*}
so that
\begin{equation}\label{e:W_est_0}
\norm{w_o}_0\leq \frac{M}{2} \delta_{q+1}^{\sfrac12}\,.
\end{equation}

\subsection{The pressure $p_1$ and the Reynolds stress $\mathring{R}_1$}
We set 
$$
\mathring{R}_1= R^0+R^1+R^2+R^3+R^4+R^5,
$$
where
\begin{align}
R^0 &= \mathcal R \left(\partial_tw+v_\ell\cdot \nabla w+w\cdot\nabla v_\ell\right)\label{e:R^0_def}\\
R^1 &=\mathcal R \div \Big(w_o \otimes w_o- \sum_l \chi_l^2 R_{\ell, l} 
-\textstyle{\frac{|w_o|^2}{2}}\Id\Big)\label{e:R^1_def}\\
R^2 &=w_o\otimes w_c+w_c\otimes w_o+w_c\otimes w_c - \textstyle{\frac{|w_c|^2 + 2\langle w_o, w_c\rangle}{3}} {\rm Id}\label{e:R^2_def}\\
R^3 &= w\otimes (v - v_\ell) + (v-v_\ell)\otimes w
 - \textstyle{\frac{2 \langle (v-v_{\ell}), w\rangle}{3}} \Id\label{e:R^3_def}\\
R^4&=\mathring R- \mathring{R}_\ell \label{e:R^4_def}\\
R^5&=\sum_l \chi_l^2 (\mathring{R}_{\ell} - \mathring{R}_{l,\ell})\label{e:R^5_def}\, .
\end{align}
Observe that $\mathring{R}_1$ is indeed a traceless symmetric tensor. The corresponding form of the new pressure will then be
\begin{equation}\label{e:def_p_1}
p_1=p-\frac{|w_o|^2}{2} - \frac{1}{3} |w_c|^2 - \frac{2}{3} \langle w_o, w_c\rangle - \frac{2}{3} \langle v-v_\ell, w\rangle \,.
\end{equation}

Recalling \eqref{e:mathringRnew} we see that $\sum_l \chi_l^2 \tr R_{\ell, l}$ is a function of time only. Since also 
$\sum_l \chi_l^2 = 1$, it is then straightforward to check that 
\begin{align*}
\div\mathring{R}_1 - \nabla p_1 &= \partial_t w + \div (v\otimes w + w\otimes v + w \otimes w) + \div \mathring{R} - \nabla p\nonumber\\
&=\partial_t w + \div (v\otimes w + w\otimes v + w \otimes w) + \partial_t v + \div (v\otimes v)\nonumber\\
&= \partial_t v_1 + \div v_1\otimes v_1\, .
\end{align*}

The following lemma will play a key role.
\begin{lemma}\label{l:doublesum}
The following identity holds:
\begin{equation}\label{e:doublesum}
w_o\otimes w_o = \sum_l \chi_l^2 R_{\ell, l} + \sum_{(k,l), (k',l'), k\neq - k'} \chi_l \chi_{l'} w_{kl} \otimes w_{k'l'}\, .
\end{equation}
\end{lemma}

\begin{proof}
Recall that the pairs $(k,l)$, $(k',l')$ are chosen so that $k\neq -k'$ if $l$ is even and 
$l'$ is odd. Moreover $\chi_l\chi_{l'} = 0$ if $l$ and $l'$ are distinct and have the same parity. 
Hence the claim follows immediately from our choice of $a_{kl}$ in \eqref{e:a_kl} and Proposition 
\ref{p:Beltrami} and Lemma \ref{l:split} (cf.\ \cite[Proposition 6.1(ii)]{DS3}).
\end{proof}

\subsection{Conditions on the parameters - hierarchy of length-scales}
In the next couple of sections we will need to estimate various expressions involving $v_\ell$ and $w$. To simplify the formulas that we arrive at, we will from now on assume the following conditions on $\mu,\lambda_{q+1}\geq 1$ and $\ell\leq 1$: 

\begin{equation}\label{e:conditions_lambdamu_2}
 \frac{\delta_{q}^{\sfrac{1}{2}}\lambda_q\ell}{\delta_{q+1}^{\sfrac{1}{2}}}\leq1, \quad
\frac{\delta_{q}^{\sfrac12}\lambda_q}{\mu} +\frac{1}{\ell\lambda_{q+1}}\leq \lambda_{q+1}^{-\beta}\quad\mbox{and}\quad
\frac{1}{\lambda_{q+1}}\leq  \frac{\delta_{q+1}^{\sfrac12}}{\mu}.
\end{equation}
These conditions imply the following orderings of length scales, which will be used to simplify the estimates in Section \ref{s:perturbation_estimates}:
\begin{equation}\label{e:ordering_params}
\frac{1}{\delta_{q+1}^{\sfrac{1}{2}}\lambda_{q+1}}\leq \frac{1}{\mu} \leq \frac{1}{\delta_q^{\sfrac{1}{2}}\lambda_{q}}\qquad
\mbox{and}\qquad \frac{1}{\lambda_{q+1}} \leq \ell \leq \frac{1}{\lambda_q}\, . 
\end{equation}
One can think of these chains of inequalities as an ordering of various length scales involved in the definition of $v_1$.

\begin{remark}
The most relevant and restrictive condition is 
$\frac{\delta_{q}^{\sfrac12}}{\mu} \leq \frac{1}{\lambda_{q}}$. 
Indeed, this condition can be thought of as a kind of CFL condition (cf.\ \cite{MR0213764}), 
restricting the coarse-grained flow to times of the order of 
$\|\nabla v\|_0^{-1}$, cf.\ Lemma \ref{l:ugly_lemma} and in particular \eqref{e:phi_l_1} below. 
Assuming only this condition on the parameters, essentially all 
the arguments for estimating the various terms would still follow through. 
The remaining inequalities are only used to simplify the many estimates needed in the rest of the
paper,
which otherwise would have a much more complicated dependence upon the various parameters. 
\end{remark}

\section{Estimates on the perturbation}\label{s:perturbation_estimates}

\begin{lemma}\label{l:ugly_lemma}
Assume \eqref{e:conditions_lambdamu_2} holds.  For $t$ in the range $\abs{\mu t-l}<1$ we have
\begin{align}
&\norm{D\Phi_l}_0\leq C\, \label{e:phi_l}\\
&\norm{D\Phi_l - \Id}_0 \leq C \frac{\delta_q^{\sfrac{1}{2}}\lambda_q}{\mu}\label{e:phi_l_1}\\
&\norm{D\Phi_l}_N\leq C \frac{\delta_q^{\sfrac{1}{2}} \lambda_q }{\mu \ell^N},& N\ge 1\label{e:Dphi_l_N}
\end{align}
Moreover,
\begin{align}
&\norm{a_{kl}}_0+\norm{L_{kl}}_0\leq C \delta_{q+1}^{\sfrac12}\label{e:L}\\
&\norm{a_{kl}}_N\leq C\delta_{q+1}^{\sfrac12}\lambda_q\ell^{1-N},&N\geq 1\label{e:Da}\\
&\norm{L_{kl}}_N\leq C\delta_{q+1}^{\sfrac12}\ell^{-N},&N\geq 1\label{e:DL}\\
&\norm{\phi_{kl}}_N\leq C \lambda_{q+1} \frac{\delta_q^{\sfrac{1}{2}} \lambda_q}{\mu \ell^{N-1}}
+ C \left(\frac{\delta_q^{\sfrac{1}{2}} \lambda_q \lambda_{q+1}}{\mu}\right)^N\nonumber \\
&\quad\qquad\;\leq C\lambda_{q+1}^{N(1-\beta)}&N\geq 1.\label{e:phi}
\end{align}
Consequently, for any $N\geq 0$
\begin{align}
&\norm{w_c}_N \leq C  \delta_{q+1}^{\sfrac12} \frac{\delta_q^{\sfrac12}\lambda_q}{\mu} 
\lambda_{q+1}^N\label{e:corrector_est},\\
&\norm{w_o}_1\leq \frac{M}{2} \delta_{q+1}^{\sfrac12} \lambda_{q+1} +
C \delta_{q+1}^{\sfrac12} \lambda_{q+1}^{1-\beta}\label{e:W_est_1}\\
&\norm{w_o}_N\leq C \delta_{q+1}^{\sfrac12}\lambda_{q+1}^N,\qquad &N\geq 2\label{e:W_est_N}
\end{align}
where the constants in \eqref{e:phi_l}-\eqref{e:phi_l_1} depend only on $M$, 
the constant in \eqref{e:Dphi_l_N} depends on $M$ and $N$, the constants in \eqref{e:L} and \eqref{e:W_est_1} 
depend on $M$ and $e$ and the remaining constants depend on $M$, $e$ and $N$.
\end{lemma}

\begin{proof} The estimates \eqref{e:phi_l} and \eqref{e:phi_l_1} are direct consequences of 
\eqref{e:Dphi_near_id} in Proposition \ref{p:transport_derivatives}, together with 
\eqref{e:ordering_params}, whereas \eqref{e:Dphi_N} in Proposition \ref{p:transport_derivatives} 
combined with the convolution estimate \eqref{e:v_ell} implies \eqref{e:Dphi_l_N}.

Next, \eqref{e:R_ell} together with  \eqref{e:max_prin},\eqref{e:trans_est_0} and 
\eqref{e:trans_est_1} in Proposition \ref{p:transport_derivatives} and \eqref{e:ordering_params} leads to
\begin{align}
&\norm{R_{\ell, l}}_0 \leq C \delta_{q+1}, &\label{e:R_0}\\ 
&\norm{R_{\ell, l}}_N \leq C \delta_{q+1} \lambda_q \ell^{1-N}, &N\geq 1.\label{e:R_N}
\end{align}
The estimate \eqref{e:L} is now a consequence of \eqref{e:R_0}, \eqref{e:phi_l_1} and \eqref{e:est_rhol}, whereas
by \eqref{e:chain0} we obtain
\begin{align}
\|a_{kl}\|_N &\leq C \delta_{q+1}^{-\sfrac{1}{2}} \|R_{\ell,l}\|_N
\leq C \delta_{q+1}^{\sfrac{1}{2}} \lambda_q\ell^{1-N}\leq C\delta_{q+1}^{\sfrac12}\ell^{-N}\, .
\end{align}
Similarly we deduce \eqref{e:DL} from
\begin{align}
\|L_{kl}\|_N \leq &C  \|a_{kl}\|_N+C\lambda_{q+1}^{-1}\|a_{kl}\|_{N+1}+\nonumber\\
& + C  \left(\|a_{kl}\|_N \|D\Phi_l -\Id\|_0 + \|a_{kl}\|_0 \|D\Phi_l\|_N\right)\, \nonumber
\end{align}
and once again using \eqref{e:ordering_params}. 

In order to prove \eqref{e:phi} we apply \eqref{e:chain1} with $m=N$ to conclude
\begin{equation*}
\norm{\phi_{kl}}_N \leq C\lambda_{q+1}\|D\Phi_l\|_{N-1}+\lambda_{q+1}^N\|D\Phi_l-\Id\|_0^N,
\end{equation*}
from which \eqref{e:phi} follows using \eqref{e:phi_l_1}, \eqref{e:Dphi_l_N} and \eqref{e:conditions_lambdamu_2}.

\medskip

Using the formula \eqref{e:corrector} together with \eqref{e:phi_l_1}, \eqref{e:Dphi_l_N}, 
\eqref{e:L} and \eqref{e:DL} we conclude
\begin{align*}
\|w_c\|_0 &\;\leq \frac{C}{\lambda_{q+1}} \|a_{kl}\|_1 +  C\|a_{kl}\|_0 \|D\Phi_l - {\rm Id}\|_0
\leq C \frac{\delta_q^{\sfrac12} \lambda_q}{\mu}
\end{align*}
and, for $N\geq 1$,
\begin{align*}
\|w_c\|_N\leq &C\sum_{kl}\chi_l\left(\frac{1}{\lambda_{q+1}}\|a_{kl}\|_{N+1}+\|a_{kl}\|_0\|D\Phi_l\|_N+\|a_{kl}\|_N\|D\Phi_l-\Id\|_0\right)\\
&+C\|w_c\|_0\sum_l\chi_l\left(\lambda_{q+1}^N\|D\Phi_l\|_0^N+\lambda_{q+1}\|D\Phi_l\|_{N-1}\right)\\
&\stackrel{\eqref{e:ordering_params}}{\leq}  
C \delta_{q+1}^{\sfrac12}\lambda_{q+1}^N
\left(\frac{\lambda_q}{\lambda_{q+1}}+\frac{\delta_q^{\sfrac12}\lambda_q}{\mu}\right)
\leq C \frac{\delta_q^{\sfrac12} \lambda_q}{\mu} \lambda_{q+1}^N\, .
\end{align*}
This proves \eqref{e:corrector_est}.
The estimates for $w_o$ follow analogously, using in addition the choice of $M$ and \eqref{e:W_est_0}.
\end{proof}


\begin{lemma}\label{l:ugly_lemma_2} Recall that $D_t= \partial_t + v_\ell \cdot \nabla$. Under the assumptions of Lemma \ref{l:ugly_lemma} we have
\begin{align}
\|D_t v_\ell\|_N &\leq C \delta_q\lambda_q\ell^{-N}\, , \label{e:Dt_v}\\
\|D_t L_{kl}\|_N &\leq C \delta_{q+1}^{\sfrac{1}{2}} \delta_q^{\sfrac{1}{2}} \lambda_q\ell^{-N}\, ,\label{e:DtL}\\
\|D^2_t L_{kl}\|_N &\leq C\delta_{q+1}^{\sfrac{1}{2}} \delta_q \lambda_q\ell^{-N-1}\, ,\label{e:D2tL}\\
\norm{D_t w_c}_N &\leq C  \delta_{q+1}^{\sfrac12}\delta_q^{\sfrac12}\lambda_q\lambda_{q+1}^N\, ,\label{e:Dt_wc}\\
\norm{D_t w_o}_N &\leq C  \delta_{q+1}^{\sfrac12}\mu\lambda_{q+1}^N\, .\label{e:Dt_wo}
\end{align}
\end{lemma}

\begin{proof}{\bf Estimate on $D_tv_\ell$. }
Note that $v_{\ell}$ satisfies the inhomogeneous transport equation
\begin{align*}
\partial_t v_{\ell} + v_{\ell}\cdot\nabla v_{\ell} &= -\nabla p*\psi_{\ell} +\div ( \mathring R_{\ell} -(v\otimes v)*\psi_{\ell} +v_{\ell}\otimes v_{\ell})\, .
\end{align*}
By hypothesis $\|\nabla p * \psi_\ell\|_N \leq C  \| p\|_1 \ell^{-N} \leq C \delta_q \lambda_q \ell^{-N}$ and analogously $\|{\rm div}\, \mathring R * \psi_\ell\| \leq C \delta_{q+1} \lambda_q \ell^{-N}$. On the other hand, by Proposition \ref{p:CET}:
\[
\|{\rm div}\, \left((v\otimes v)*\psi_\ell - v_\ell \otimes v_\ell\right)\|_N \leq C \ell^{1-N} \|v\|_1^2 \leq C \ell^{1-N} \delta_q \lambda_q^2\, .
\]
Thus \eqref{e:Dt_v} follows from \eqref{e:ordering_params}.

\medskip

\noindent{\bf Estimates on $L_{kl}$.} 
Recall that $L_{kl}$ is defined as
$$
L_{kl}:=a_{kl}B_k+\Bigl(\frac{i}{\lambda_{q+1}}\nabla a_{kl}-a_{kl}(D\Phi_{l}-\Id)k\Bigr)\times\frac{k\times B_k}{|k|^2}\,.
$$
Using that 
\begin{equation}\label{e:someidentities}
\begin{split}
D_ta_{kl}&=0,\quad D_t\Phi_l=0,\\
D_t\nabla a_{kl}&=-Dv_{\ell}^T\nabla a_{kl},\quad D_tD\Phi_l=-D\Phi_lDv_\ell,
\end{split}
\end{equation}
we obtain
$$
D_tL_{kl}=\left(-\frac{i}{\lambda_{q+1}}Dv_\ell^T\nabla a_{kl}+a_{kl}D\Phi_lDv_\ell k\right)\times\frac{k\times B_k}{|k|^2}.
$$
Consequently, for times $|t-l|<\mu^{-1}$ and $N\geq 0$ we have
\begin{align*}
\|D_tL_{kl}\|_N&\leq C\delta_{q+1}^{\sfrac12}\delta_q^{\sfrac12}\lambda_q\ell^{-N}\left(\frac{\lambda_q}{\lambda_{q+1}}+\lambda_q\ell+\frac{\delta_q^{\sfrac12}\lambda_q}{\mu}+1\right)\\
&\leq C\delta_{q+1}^{\sfrac12}\delta_q^{\sfrac12}\lambda_q\ell^{-N}\, ,
\end{align*}
where we have used \eqref{e:Holderproduct}, Lemma \ref{l:ugly_lemma} and \eqref{e:ordering_params}.
Taking one more derivative and using \eqref{e:someidentities} again, we obtain
\begin{align*}
D_t^2L_{kl}=&\Bigl(-\frac{i}{\lambda_{q+1}}(D_tDv_\ell)^T\nabla a_{kl}+\frac{i}{\lambda_{q+1}}Dv_\ell^TDv_\ell^T\nabla a_{kl}+\\
&-a_{kl}D\Phi_lDv_\ell Dv_\ell k+a_{kl}D\Phi_lD_tDv_\ell k\Bigr)\times\frac{k\times B_k}{|k|^2}.
\end{align*}
Note that $D_tDv_\ell=DD_tv_\ell-Dv_\ell Dv_\ell$, so that
\begin{align*}
\|D_tDv_\ell\|_N&\leq \|D_tv_\ell\|_{N+1}+\|Dv_\ell\|_N\|Dv_\ell\|_0\\
&\leq C\delta_q\lambda_q\ell^{-N-1}\left(1+\lambda_q\ell\right)\leq C\delta_q\lambda_q\ell^{-N-1}.
\end{align*}
It then follows from the product rule \eqref{e:Holderproduct} and \eqref{e:ordering_params} that
\begin{align*}
\|D_t^2L_{kl}\|_N&\leq C\delta_{q+1}^{\sfrac{1}{2}}\delta_q\lambda_q\ell^{-N-1}\left(\frac{\lambda_q}{\lambda_{q+1}}+\frac{\lambda_q^2\ell}{\lambda_{q+1}}+\lambda_q\ell+(\lambda_q\ell)^2+\frac{\delta_q^{\sfrac12}\lambda_q}{\mu}+1\right)\\
&\leq C\delta^{\sfrac{1}{2}}_{q+1}\delta_q\lambda_q\ell^{-N-1}.
\end{align*}

\medskip

\noindent{\bf Estimates on $w_c$.} Observe that
$w_c = \sum \chi_l (L_{kl}-a_{kl} B_k) e^{i\lambda_{q+1} k\cdot \Phi_l}$ (see
\eqref{e:corrector} and \eqref{e:L_kl}). Differentiating this identity we
then conclude
\begin{align*}
D_t w_c&=\sum_{kl}\chi_l\,(D_t L_{kl})\,e^{i\lambda_{q+1}k\cdot\Phi_l}+(\partial_t\chi_l)\, (L_{kl}-a_{kl}B_k)\,e^{i\lambda_{q+1}k\cdot\Phi_l}\\
&=\sum_{kl} \chi_l (D_t L_{kl}) \phi_{kl}e^{i\lambda_{q+1} k \cdot x}+\\
&+\sum_{kl} (\partial_t \chi_l) \left(\frac{i\nabla a_{kl} }{\lambda_{q+1}}  - a_{kl} \left(D\Phi_l - \Id\right) k\right)\times \frac{k\times B_k}{|k|^2} \phi_{kl} e^{i\lambda_{q+1} k \cdot x}\,.
\end{align*}
Hence we obtain \eqref{e:Dt_wc} as a consequence of Lemma \ref{l:ugly_lemma} and \eqref{e:DtL}.

\medskip

\noindent{\bf Estimates on $w_o$.} Using \eqref{e:someidentities} we have
\[
D_t w_o = \sum_{k,l} \chi_l' a_{kl}\phi_{kl} e^{i\lambda_{q+1} k\cdot x}\, .
\]
Therefore \eqref{e:Dt_wo} follows immediately from Lemma \ref{l:ugly_lemma}.
\end{proof}

\section{Estimates on the energy}\label{s:estimates_energy}

\begin{lemma}[Estimate on the energy]\label{l:energy}
\begin{equation}\label{e:energy}
\left|e(t)(1-\delta_{q+2})-\int_{\T^3}|v_1|^2\,dx\right| \leq \frac{1}{\mu} + 
C \frac{\delta_{q+1}\delta_q^{\sfrac{1}{2}}\lambda_q}{\mu} + C\frac{\delta_{q+1}^{\sfrac{1}{2}} \delta_q^{\sfrac{1}{2}}\lambda_q}{\lambda_{q+1}}\, .
\end{equation}
\end{lemma}
\begin{proof} Define
$$
\bar{e}(t):=3(2\pi)^3\sum_l\chi_l^2(t)\rho_l.
$$
Using Lemma \ref{l:doublesum} we then have
\begin{align}
|w_o|^2 &=\sum_l \chi_l^2\tr R_{\ell,l} + \sum_{(k,l), (k', l'), k\neq -k^{\prime}} \chi_l\chi_{l^{\prime}} w_{kl} \cdot w_{k,l'}\nonumber\\
&= (2\pi)^{-3}\bar{e} + \sum_{(k,l), (k',l'), k\neq - k'} \chi_l \chi_{k'} a_{kl} a_{k'l'} \phi_{kl} \phi_{k'l'} e^{i\lambda_{q+1} (k+k') \cdot x}\, .
\end{align}
Observe that $\bar{e}$ is a function of $t$ only and that, since $(k+k')\neq 0$ in the sum above, we can apply
Proposition \ref{p:stat_phase}(i) with $m=1$. From Lemma \ref{l:ugly_lemma} we then deduce
\begin{align}
& \left|\int_{\T^3}|w_o|^2\,dx-\bar{e} (t)\right|\leq C\frac{\delta_{q+1} \delta_q^{\sfrac{1}{2}} \lambda_q}{\mu}  +  C \frac{\delta_{q+1}\lambda_q}{\lambda_{q+1}} \, .\label{e:energy1}
\end{align}
Next we recall \eqref{e:w_compact_form}, integrate by parts and use \eqref{e:L} and \eqref{e:phi} to reach
\begin{equation}\label{e:energy2}
\left|\int_{\T^3}v\cdot w\,dx\right|\leq C \frac{\delta_{q+1}^{\sfrac{1}{2}} \delta_q^{\sfrac{1}{2}}\lambda_q}{\lambda_{q+1}}\, .
\end{equation}
Note also that by \eqref{e:corrector_est} we have
\begin{align}
\int_{\T^3}|w_c|^2+\abs{w_c w_o}\,dx&\leq C\frac{\delta_{q+1} \delta_q^{\sfrac{1}{2}} \lambda_q}{\mu}\,  .\label{e:energy3}
\end{align}
Summarizing, so far we have achieved
\begin{align}
&\left|\int_{\T^3} |v_1|^2 \,dx- \left(\bar{e} (t) +\int_{\T^3} |v|^2\,dx\right)\right| \stackrel{\eqref{e:energy2}}{\leq} 
\left|\int_{\T^3} |w|^2\,dx - \bar{e} (t)\right| + C \frac{\delta_{q+1}^{\sfrac{1}{2}} \delta_q^{\sfrac{1}{2}}\lambda_q}{\lambda_{q+1}}\nonumber\\
&\stackrel{\eqref{e:energy3}}{\leq} \left|\int_{\T^3} |w_o|^2 \,dx- \bar{e} (t)\right| + C \frac{\delta_{q+1}^{\sfrac{1}{2}} \delta_q^{\sfrac{1}{2}}\lambda_q}{\lambda_{q+1}} + C\frac{\delta_{q+1} \delta_q^{\sfrac{1}{2}} \lambda_q}{\mu}\nonumber\\
&\stackrel{\eqref{e:energy1}}{\leq} C \frac{\delta_{q+1}^{\sfrac{1}{2}} \delta_q^{\sfrac{1}{2}}\lambda_q}{\lambda_{q+1}} + C\frac{\delta_{q+1} \delta_q^{\sfrac{1}{2}} \lambda_q}{\mu}\,.\label{e:energy10}
\end{align}

Next, recall that
\begin{align*}
\bar{e} (t) &= 3 (2\pi)^3\sum_l \chi_l^2  \rho_l\\
&=(1-\delta_{q+2})\sum_l \chi_l^2 e\left(\frac{\mu}{l}\right)-\sum_l\chi_l^2\int_{\T^3}|v(x,l\mu^{-1})|^2\,dx\, .
\end{align*}
Since $\abs{t-\frac{l}{\mu}}<\mu^{-1}$ on the support of $\chi_l$ and since $\sum_l \chi_l^2 =1$, we have
\[
\left|e (t) - \sum_l \chi_l^2 e\left(\frac{l}{\mu}\right)\right| \leq  \mu^{-1}\, .
\]
Moreover, using the Euler-Reynolds equation, we can compute
\begin{align}
&\int_{\T^3} \left(|v(x,t)|^2-\left|v\left(x,l\mu^{-1}\right)\right|^2\right)\, dx = \int_{\frac{l}{\mu}}^t \int_{\T^3} \partial_t |v|^2\nonumber\\
=&- \int_{\frac{l}{\mu}}^t \int_{\T^3} {\rm div}\, \left(v \left(|v|^2 + 2p\right)\right) + 2 \int_{\frac{l}{\mu}}^t \int_{\T^3} v \cdot \div \mathring{R}
= -   2 \int_{\frac{l}{\mu}}^t \int_{\T^3} Dv: \mathring{R}\, .\nonumber
\end{align}
Thus, for $\left|t-\frac{l}{\mu}\right|\leq \mu^{-1}$ we conclude
\[
\left|\int_{\T^3} |v(x,t)|^2-\left|v(x,l\mu^{-1})\right|^2\, dx \right| \leq C \frac{\delta_{q+1}\delta_q^{\sfrac{1}{2}}\lambda_q}{\mu}\, .
\]
Using again $\sum \chi_l^2=1$, we then conclude
\begin{equation}\label{e:energy4}
\left|e(t) (1-\delta_{q+2}) - \left(\bar{e} (t)+ \int_{\T^3} |v(x,t)|^2 \, dx\right)\right| \leq \frac{1}{\mu} + 
C \frac{\delta_{q+1}\delta_q^{\sfrac{1}{2}}\lambda_q}{\mu}\, .
\end{equation}
The desired conclusion \eqref{e:energy} follows from \eqref{e:energy10} and \eqref{e:energy4}.
\end{proof}

\section{Estimates on the Reynolds stress}\label{s:reynolds}

In this section we bound the new Reynolds Stress $\mathring R_1$. The general pattern in estimating derivatives of the Reynolds stress is that: 
\begin{itemize}
\item the space derivative gets an extra factor of $\lambda_{q+1}$ (when the derivative falls on the exponential factor), 
\item the transport derivative gets an extra factor $\mu$ (when the derivative falls on the time cut-off). 
\end{itemize}
In fact the transport derivative is slightly more subtle, because in $R^0$ a second transport derivative of the perturbation $w$ appears, which leads to an additional term (see \eqref{e:extraterm}). Nevertheless, we organize the estimates in the following proposition according to the above pattern.

\begin{proposition}\label{p:R} For any choice of small positive numbers $\eps$ and $\beta$, there is a constant $C$ (depending only upon these parameters and on $e$ and $M$) such that, if $\mu$, $\lambda_{q+1}$ and $\ell$ satisfy the conditions \eqref{e:conditions_lambdamu_2}, then we have
\begin{align}
\|R^0\|_0 +\frac{1}{\lambda_{q+1}}\|R^0\|_1+\frac{1}{\mu}\|D_tR^0\|_0&\leq C\frac{\delta_{q+1}^{\sfrac{1}{2}} \mu}{\lambda_{q+1}^{1-\eps}}+\frac{\delta_{q+1}^{\sfrac12}\delta_q\lambda_q}{\lambda_{q+1}^{1-\eps}\mu\ell}\label{e:R0}\\
\|R^1\|_0 +\frac{1}{\lambda_{q+1}}\|R^1\|_1+\frac{1}{\mu}\|D_tR^1\|_0&\leq C \frac{\delta_{q+1}\delta_q^{\sfrac{1}{2}} \lambda_q \lambda_{q+1}^\eps}{\mu}\label{e:R1}\\
\|R^2\|_0 +\frac{1}{\lambda_{q+1}}\|R^2\|_1+\frac{1}{\mu}\|D_tR^2\|_0&\leq C \frac{\delta_{q+1}\delta_q^{\sfrac{1}{2}} \lambda_q}{\mu} \label{e:R2}\\
\|R^3\|_0 +\frac{1}{\lambda_{q+1}}\|R^3\|_1+\frac{1}{\mu}\|D_tR^3\|_0&\leq C \delta_{q+1}^{\sfrac{1}{2}} \delta_q^{\sfrac{1}{2}} \lambda_q \ell\label{e:R3}\\
\|R^4\|_0 +\frac{1}{\lambda_{q+1}}\|R^4\|_1+\frac{1}{\mu}\|D_tR^4\|_0&\leq C \frac{\delta_{q+1} \delta_q^{\sfrac{1}{2}}\lambda_q}{\mu} + C\delta_{q+1} \lambda_q \ell\label{e:R4}\\
\|R^5\|_0 +\frac{1}{\lambda_{q+1}}\|R^5\|_1+\frac{1}{\mu}\|D_tR^5\|_0&\leq C \frac{\delta_{q+1} \delta_q^{\sfrac{1}{2}}\lambda_q}{\mu}\, .\label{e:R5}
\end{align}
Thus
\begin{equation}
\begin{split}
&\|\mathring{R}_1\|_0+\frac{1}{\lambda_{q+1}}\|\mathring{R}_1\|_1 +\frac{1}{\mu}\|D_t\mathring{R}_1\|_0\leq \\
&\leq C \left(\frac{\delta_{q+1}^{\sfrac{1}{2}} \mu}{\lambda_{q+1}^{1-\eps}} + \frac{\delta_{q+1} \delta_q^{\sfrac{1}{2}} \lambda_q \lambda_{q+1}^\eps}{\mu} +
\delta_{q+1}^{\sfrac{1}{2}} \delta_q^{\sfrac{1}{2}} \lambda_q \ell+\frac{\delta_{q+1}^{\sfrac12}\delta_q\lambda_q}{\lambda_{q+1}^{1-\eps}\mu\ell}\right)\, ,\label{e:allR}
\end{split}
\end{equation}
and, moreover,
\begin{multline}
\|\partial_t \mathring{R}_1 + v_1\cdot \nabla \mathring{R}_1\|_0\leq \\
\leq C \delta_{q+1}^{\sfrac{1}{2}}\lambda_{q+1} \left(\frac{\delta_{q+1}^{\sfrac{1}{2}} \mu}{\lambda_{q+1}^{1-\eps}} 
+ \frac{\delta_{q+1} \delta_q^{\sfrac{1}{2}} \lambda_q \lambda_{q+1}^\eps}{\mu} 
+ \delta_{q+1}^{\sfrac{1}{2}} \delta_q^{\sfrac{1}{2}} \lambda_q \ell+\frac{\delta_{q+1}^{\sfrac12}\delta_q\lambda_q}{\lambda_{q+1}^{1-\eps}\mu\ell}\right).\label{e:Dt_R_all}
\end{multline}
\end{proposition}

\begin{proof}{\bf Estimates on $R^0$.}  We start by calculating
\[
\partial_tw+v_\ell\cdot \nabla w+w\cdot \nabla v_\ell=\sum_{kl}
\left(\chi_l'L_{kl}+\chi_lD_tL_{kl}+\chi_lL_{kl}\cdot \nabla v_\ell\right)e^{ik\cdot\Phi_l}\,.
\]
Define $\Omega_{kl}:= \left(\chi_l'L_{kl}+\chi_lD_tL_{kl}+\chi_lL_{kl}\cdot \nabla v_\ell\right) \phi_{kl}$
and write (recalling the identity $\phi_{kl}e^{i\lambda_{q+1}k\cdot x}=e^{i\lambda_{q+1}k\cdot\Phi_l}$),
\begin{equation}\label{e:R_0_10}
\partial_tw+v_\ell\cdot \nabla w+w\cdot \nabla v_\ell=\sum_{kl}\Omega_{kl}e^{i\lambda_{q+1}k\cdot x}\, .
\end{equation}
Using Lemmas \ref{l:ugly_lemma} and \ref{l:ugly_lemma_2} and \eqref{e:ordering_params}
\begin{align*}
\|\Omega_{kl}\|_0&\leq C \delta_{q+1}^{\sfrac12}\mu\left(1+\frac{\delta_q^{\sfrac12}\lambda_q}{\mu}\right)
\leq C \delta_{q+1}^{\sfrac12}\mu
\end{align*}
and similarly, for $N\geq 1$
\begin{align*}
\|\Omega_{kl}\|_N&\leq C\delta_{q+1}^{\sfrac12}\mu\left(\ell^{-N}+\|\phi_{kl}\|_N\right)
\leq C\delta_{q+1}^{\sfrac12}\mu\,\lambda_{q+1}^{N(1-\beta)}.
\end{align*}
Moreover, observe that although this estimate has been derived for $N$ integer, by the interpolation inequality
\eqref{e:Holderinterpolation2} it can be easily extended to any real $N\geq 1$ (besides, this fact will be used frequently in the rest of the proof).
Applying Proposition \ref{p:stat_phase}(ii) we obtain
\begin{align}
&\|R_0\|_0 \leq \sum_{kl} \left(\lambda_{q+1}^{\eps-1} \|\Omega_{kl}\|_0 + \lambda_{q+1}^{-N+\eps} [\Omega_{kl}]_N + 
\lambda_{q+1}^{-N} [\Omega_{kl}]_{N+\eps}\right)\nonumber\\
\leq & C\delta^{\sfrac{1}{2}}_{q+1} \mu\,\left(\lambda_{q+1}^{-1+\eps}+\lambda_{q+1}^{-N\beta+\eps}\right)
\end{align}
It suffices to choose $N$ so that $N\beta \geq 1$ in order to achieve 
$$
\|R^0\|_0 \leq C \delta_{q+1}^{\sfrac{1}{2}} \mu \lambda_{q+1}^{\eps-1}.
$$
As for $\|R^0\|_1$, we differentiate \eqref{e:R_0_10}. We therefore conclude
\[
\partial_j R^0 = \RR \left( \sum_{kl} (i \lambda_{q+1} k_j\Omega_{kl} + \partial_j \Omega_{kl} ) e^{i\lambda_{q+1}k\cdot x} \right)\, .
\]
Applying Proposition \ref{p:stat_phase}(ii) as before we conclude $\|R^0\|_1 \leq C \delta_{q+1}^{\sfrac{1}{2}} \mu \lambda_{q+1}^{\eps}$.

\medskip

\noindent{\bf Estimates on $D_tR^0$.} We start by calculating
\begin{align*}
&D_t \left(\partial_tw+v_\ell\cdot \nabla w+w\cdot \nabla v_\ell\right)=
\sum_{kl}  \Bigl(\partial_t^2\chi_lL_{kl}+2\partial_t\chi_lD_tL_{kl}+\chi_lD_t^2L_{kl}+\\
&+\partial_t\chi_lL_{kl}\cdot\nabla v_\ell+\chi_lD_tL_{kl}\cdot\nabla v_\ell+\chi_lL_{kl}
\cdot\nabla D_tv_\ell-\chi_lL_{kl}\cdot\nabla v_\ell\cdot \nabla v_\ell\Bigr)e^{ik\cdot\Phi_l}\\
& =:\sum_{kl}\Omega'_{kl}e^{i\lambda_{q+1}k\cdot x}.
\end{align*}
As before, we have
\begin{align}\label{e:extraterm}
\|\Omega'_{kl}\|_0&\leq C\delta_{q+1}^{\sfrac12}\mu\left(\mu+\frac{\delta_q\lambda_q}{\mu\ell}+\delta_q^{\sfrac{1}{2}}\lambda_q + \frac{\delta_q\lambda_q^2}{\mu}\right)
\leq C\delta_{q+1}^{\sfrac12}\left(\mu+\frac{\delta_q\lambda_q}{\mu\ell}\right)
\end{align}
and, for any $N\geq 1$
\begin{align*}
\|\Omega'_{kl}\|_N&\leq C\delta_{q+1}^{\sfrac12}\mu\ell^{-N}\left(\mu+\frac{\delta_q\lambda_q}{\mu\ell}+\delta_q^{\sfrac12}\lambda_q+\frac{\delta_q\lambda_q^2}{\mu}\right)+\|\Omega'_{kl}\|_0\|\phi_{kl}\|_N\\
&\leq C\delta_{q+1}^{\sfrac12}\mu\left(\mu+\frac{\delta_q\lambda_q}{\mu\ell}\right)\left(\ell^{-N}+\|\phi_{kl}\|_N\right)\\
&\leq C\delta_{q+1}^{\sfrac12}\mu\left(\mu+\frac{\delta_q\lambda_q}{\mu\ell}\right)\,\lambda_{q+1}^{N(1-\beta)}\,.
\end{align*}

Next, observe that we can write
\begin{align*}
&D_tR^0=\Bigl([D_t,\mathcal{R}]+\mathcal{R}D_t\Bigr)(\partial_tw+v_\ell\cdot \nabla w+w\cdot \nabla v_\ell)\\
=&\Bigl([v_\ell,\mathcal{R}]\nabla+\mathcal{R}D_t\Bigr)(\partial_tw+v_\ell\cdot \nabla w+w\cdot \nabla v_\ell)\\
=&\sum_{kl}[v_\ell,\mathcal{R}](\nabla\Omega_{kl}e^{i\lambda_{q+1}k\cdot x})+ i\lambda_{q+1}[v_\ell\cdot k,\mathcal{R}](\Omega_{kl}e^{i\lambda_{q+1}k\cdot x})+\mathcal{R}(\Omega'_{kl}e^{i\lambda_{q+1}k\cdot x})\, 
\end{align*}
(where, as it is customary, $[A,B]$ denotes the commutator $AB - BA$ of two operators $A$ and $B$).

Using the estimates for $\Omega'_{kl}$ derived above, and applying Proposition \ref{p:stat_phase}(ii), we obtain
\begin{align*}
\|\mathcal{R}&(\Omega'_{kl}e^{i\lambda_{q+1}k\cdot x})\|_0 \leq\frac{\|\Omega'_{kl}\|_0}{ \lambda_{q+1}^{1-\eps}} + \frac{[\Omega'_{kl}]_N}{\lambda_{q+1}^{N-\eps} }+ \frac{[\Omega'_{kl}]_{N+\eps}}{\lambda_{q+1}^{N} }\\
&\leq  C\frac{\delta^{\sfrac{1}{2}}_{q+1} \mu}{\lambda_{q+1}^{1-\eps}}\,\left(\mu+\frac{\delta_q\lambda_q}{\mu\ell}\right)\left(1+\lambda_{q+1}^{1-N\beta}\right)
\end{align*}

Furthermore, applying Proposition \ref{p:commutator} we obtain
\begin{align*}
&\left\|[v_\ell,\mathcal{R}](\nabla\Omega_{kl}e^{i\lambda_{q+1}k\cdot x})\right\|_0\leq 
C\frac{\|\Omega_{kl}\|_{N+1+\eps}\|v_\ell\|_{2+\eps} +\|\Omega_{kl}\|_{3+\eps}\|v_\ell\|_{N+\eps}}{\lambda_{q+1}^{N}}\nonumber\\
&\qquad + C\frac{\|\Omega_{kl}\|_{N+1}\|v_\ell\|_{2}+\|\Omega_{kl}\|_{3}\|v_\ell\|_N}{\lambda_{q+1}^{N-\eps}} + C
 \frac{\|\Omega_{kl}\|_1\|v_\ell\|_1}{\lambda_{q+1}^{2-\eps}}\\
 &\leq C\frac{\delta_{q+1}^{\sfrac12}\mu}{\lambda_{q+1}^{1-\eps}}\delta_q^{\sfrac12}\lambda_q\left(1+\lambda_{q+1}^{3-\beta-\beta N}\right)
 \leq C\frac{\delta_{q+1}^{\sfrac12}\mu^2}{\lambda_{q+1}^{1-\eps}}
 \left(1+\lambda_{q+1}^{3-\beta(N+1)}\right)
\end{align*}
and similarly
$$
\lambda_{q+1}\left\|[ v_\ell\cdot k,\mathcal{R}](\Omega_{kl}e^{i\lambda_{q+1}k\cdot x})\right\|_0\leq C\frac{\delta_{q+1}^{\sfrac12}\mu^2}{\lambda_{q+1}^{1-\eps}}
 \left(1+\lambda_{q+1}^{3-\beta(N+1)}\right).
$$

By choosing $N\in\N$ sufficiently large so that 
$$
\beta(N+1)\geq 3\quad\textrm{and }\quad\beta N\geq 1,
$$
we deduce 
$$
\|D_tR^0\|_0 \leq C\frac{\delta_{q+1}^{\sfrac{1}{2}} \mu}{\lambda_{q+1}^{1-\eps}}\left(\mu+\frac{\delta_q\lambda_q}{\mu\ell}\right)
$$
This concludes the proof of \eqref{e:R0}.

\medskip

\noindent{\bf Estimates on $R^1$.} 
Using Lemma \ref{l:doublesum} we have
\begin{align*}
\div&\left(w_o\otimes w_o - \sum_l \chi_l^2\mathring{R}_{\ell, l}-\frac{|w_o|^2}{2}\Id\right)=\\
=&\mathop{\sum_{(k,l),(k',l')}}_{k+k'\neq 0}\chi_l\chi_{l'}\div\left(w_{kl}\otimes w_{k'l'} - \frac{w_{kl}\cdot w_{k'l'}}{2} \Id\right) = I + II
\end{align*}
where, setting $f_{klk'l'}:=\chi_l\chi_{l'}a_{kl}a_{k'l'}\phi_{kl}\phi_{k'l'}$,
\begin{align*}
I=&\mathop{\sum_{(k,l),(k',l')}}_{k+k'\neq 0}\left(B_k\otimes B_{k'}-\tfrac{1}{2}(B_k\cdot B_{k'})\Id\right)\nabla f_{klk'l'}e^{i\lambda_{q+1}(k+k')\cdot x}\\
II=&i\lambda_{q+1}\mathop{\sum_{(k,l),(k',l')}}_{k+k'\neq 0}f_{klk'l'}\left(B_k\otimes B_{k'}-\tfrac{1}{2}(B_k\cdot B_{k'})\Id\right)(k+k')e^{i\lambda_{q+1}(k+k')\cdot x}\, .
\end{align*}
Concerning $II$, recall that the summation is over all $l\in\Z\cap[0,\mu]$ and all $k\in\Lambda^e$ if $l$ is even and all $k\in\Lambda^o$ if $l$ is odd. Furthermore, both $\Lambda^e,\Lambda^o\subset\bar{\lambda}S^2\cap\Z^3$ satisfy the conditions of Lemma \ref{l:split}. Therefore we may symmetrize the summand in $II$ in $k$ and $k'$. On the other hand, recall from Lemma \ref{l:BkBk'} that 
$$
(B_k\otimes B_{k'}+B_{k'}\otimes B_k)(k+k')=(B_k\cdot B_{k'})(k+k').
$$
From this we deduce that $II=0$.

Concerning $I$, we first note, using the product rule, \eqref{e:L} and \eqref{e:Da}, that
\begin{equation*}
[f_{klk'l'}]_N\leq C\delta_{q+1}\left(\lambda_q\ell^{1-N}+\|\phi_{kl}\phi_{k'l'}\|_N\right)
\quad \mbox{for $N\geq 1$}\,.
\end{equation*}
By Lemma \eqref{e:phi} and \eqref{e:conditions_lambdamu_2} (cf.\ \eqref{e:ordering_params})
we then conclude
\begin{align*}
[f_{klk'l'}]_1&\leq C\delta_{q+1}\left(\lambda_q+\lambda_{q+1}\frac{\delta_q^{\sfrac12}\lambda_q}{\mu}\right)
\leq C\delta_{q+1}\lambda_{q+1}\frac{\delta_q^{\sfrac12}\lambda_q}{\mu}\,,\\
[f_{klk'l'}]_N&\leq C\delta_{q+1}\lambda_{q+1}^{N(1-\beta)},\qquad N\geq 2.
\end{align*}
Applying Proposition \ref{p:stat_phase}(ii) to $I$ 
we obtain
\begin{align}
\|R^1\|_0 &\leq \mathop{\sum_{(k,l),(k',l')}}_{k+k'\neq 0} \left(\lambda_{q+1}^{\eps-1} [f_{klk'l'}]_1 + \lambda_{q+1}^{-N+\eps} [f_{klk'l'}]_{N+1} + 
\lambda_{q+1}^{-N} [f_{klk'l'}]_{N+1+\eps}\right)\nonumber\\
&\leq C\delta_{q+1}\,\left(\lambda_{q+1}^{\eps}\frac{\delta_q^{\sfrac12}\lambda_q}{\mu}+\lambda_{q+1}^{1-N\beta+\eps}\right)
\end{align}
By choosing $N$ sufficiently large we deduce 
$$
\|R^1\|_0 \leq C \frac{\delta_{q+1}\delta_q^{\sfrac{1}{2}} \lambda_q \lambda_{q+1}^\eps}{\mu}
$$
as required. Moreover, differentiating we conclude $\partial_j R^1 = \RR (\partial_j I)$ where
\begin{align}
\partial_j I &= \mathop{\sum_{(k,l),(k',l')}}_{k+k'\neq 0}\left(B_k\otimes B_{k'}-\tfrac{1}{2}(B_k\cdot B_{k'})\Id\right)\cdot\nonumber\\
&\qquad\qquad\qquad(i \lambda_{q+1} (k+k')_j \nabla f_{klk'l'}+ \partial_j\nabla f_{klk'l'} ) e^{i\lambda_{q+1}
(k+k')\cdot x}\, .
\end{align}
Therefore we apply again Proposition \ref{p:stat_phase}(ii) to conclude the desired estimate
for $\|R^1\|_1$.

\medskip

\noindent{\bf Estimates on $D_tR^1$.}
As in the estimate for $D_tR^0$, we again make use of the identity
$D_t\mathcal{R}=[v_{\ell},\mathcal{R}]\nabla+\mathcal{R}D_t$ in order to write
\begin{align*}
 &D_tR^1=\mathop{\sum_{(k,l),(k',l')}}_{k+k'\neq 0}\left([v_\ell,\mathcal{R}]\left(\nabla U_{klk'l'}e^{i\lambda_{q+1}(k+k')\cdot x}\right)\right.\\
&+i\lambda_{q+1}[v_\ell\cdot (k+k'),\mathcal{R}]\left(U_{klk'l'}e^{i\lambda_{q+1}(k+k')\cdot x}\right)
+\left.\mathcal{R}\left(U'_{klk'l'}e^{i\lambda_{q+1}(k+k')\cdot x}\right)\right),
\end{align*}
where we have set 
$U_{klk'l'} = \left(B_k\otimes B_{k'}-\tfrac{1}{2}(B_k\cdot B_{k'})\Id\right)\nabla f_{klk'l'}$ and
\begin{align*}
D_t\div\bigl(w_o\otimes w_o - \sum_l \chi_l^2\mathring{R}_{\ell, l}- \frac{|w_o|^2}{2}\Id\bigr)&=\mathop{\sum_{(k,l),(k',l')}}_{k+k'\neq 0}U'_{klk'l'}e^{i\lambda_{q+1}(k+k')\cdot x}.
\end{align*}
In order to further compute $U'_{klk'l'}$, we write
\begin{align*}
\nabla &f_{klk'l'}\,e^{i\lambda_{q+1}(k+k')\cdot x}=\chi_l\chi_{l'}\left(a_{kl}\nabla a_{k'l'}+a_{k'l'}\nabla a_{kl}\right)e^{i\lambda_{q+1}(k\cdot\Phi_l+k'\cdot\Phi_{l'})}
+\\
&+i\lambda_{q+1}\chi_l\chi_{l'}a_{kl}a_{k'l'}\left((D\Phi_l-\Id)k+(D\Phi_{l'}-\Id)k'\right)e^{i\lambda_{q+1}(k\cdot\Phi_l+k'\cdot\Phi_{l'})}
\end{align*}
and hence, using \eqref{e:someidentities},
\begin{align*}
D_t&\left(\nabla f_{klk'l'}\,e^{i\lambda_{q+1}(k+k')\cdot x}\right)=(\chi_l\chi_{l'})'\left(a_{kl}\nabla a_{k'l'}+a_{k'l'}\nabla a_{kl}\right)e^{i\lambda_{q+1}(k\cdot\Phi_l+k'\cdot\Phi_{l'})}\\
&+ i\lambda_{q+1}(\chi_l\chi_{l'})'a_{kl}a_{k'l'}\left((D\Phi_l-\Id)k+(D\Phi_{l'}-\Id)k'\right)e^{i\lambda_{q+1}(k\cdot\Phi_l+k'\cdot\Phi_{l'})}\\
&-\chi_l\chi_{l'}\left(a_{kl}Dv_\ell^T\nabla a_{k'l'}+a_{k'l'}Dv_\ell^T\nabla a_{kl}\right)e^{i\lambda_{q+1}(k\cdot\Phi_l+k'\cdot\Phi_{l'})}\\
&-\chi_l\chi_{l'}a_{kl}a_{k'l'}\left(D\Phi_lDv_\ell^Tk+D\Phi_{l'}Dv_\ell^Tk'\right)e^{i\lambda_{q+1}(k\cdot\Phi_l+k'\cdot\Phi_{l'})}\\
&  = : \left(\Sigma^1_{klk'l'} + \Sigma^2_{klk'l'} + \Sigma^3_{klk'l'} + \Sigma^4_{klk'l'}\right)
e^{i\lambda_{q+1}(k\cdot\Phi_l+k'\cdot\Phi_{l'})} \\
&=: \Sigma_{klk'l'} e^{i\lambda_{q+1}(k\cdot\Phi_l+k'\cdot\Phi_{l'})} \, .
\end{align*}
Ignoring the subscripts we can use
\eqref{e:Holderproduct}, Lemma \ref{l:ugly_lemma} and Lemma \ref{l:ugly_lemma_2} to estimate
\begin{align*}
\|\Sigma\|_N &\leq C\delta_{q+1}\lambda_q \ell^{-N}(\mu+\lambda_{q+1}\delta_q^{\sfrac12}+\delta_q^{\sfrac12}
\lambda_q + \delta_q^{\sfrac12})
\stackrel{\eqref{e:ordering_params}}{\leq} 
C\delta_{q+1} \lambda_{q+1}\delta_q^{\sfrac12}\lambda_q \ell^{-N}\, .
\end{align*}
We thus conclude
\begin{align*}
\|U'_{klk'l'}\|_N&\leq C \|\Sigma_{klk'l'}\|_N + C \|\Sigma_{klk'l'}\|_0\|\phi_{kl}\phi_{k'l'}\|_N\nonumber\\
&\leq C \delta_{q+1} \lambda_{q+1} \delta_q^{\sfrac12} \lambda_q 
\left( \ell^{-N} + \lambda_{q+1}^{N(1-\beta)}\right) \leq C \delta_{q+1} \delta_q^{\sfrac12} \lambda_q \lambda_{q+1}^{1+ N (1-\beta)}
\end{align*}
The estimate on $\|D_tR^1\|_0$ now follows exactly as above for $D_tR^0$ applying Proposition \ref{p:commutator} to the commutator terms. This concludes the verification of \eqref{e:R1}.

\medskip

\noindent{\bf Estimates on $R^2$ and $D_t R^2$.} 
Using Lemma \ref{l:ugly_lemma} we have
\begin{align*}
\|R^2\|_0 &\leq C(\|w_c\|^2_0 + \|w_o\|_0 \|w_c\|_0)
\leq C \delta_{q+1}\frac{\delta_q^{\sfrac12}\lambda_q}{\mu}\, ,\\
\|R^2\|_1 &\leq C(\|w_c\|_1 \|w_c\|_0 + \|w_o\|_1 \|w_c\|_0 + \|w_o\|_0 \|w_c\|_1)
\leq C \delta_{q+1} \frac{\delta_q^{\sfrac12}\lambda_q}{\mu}\lambda_{q+1}\,.
\end{align*}
Similarly, with the Lemmas \ref{l:ugly_lemma} and \ref{l:ugly_lemma_2} we achieve
\begin{align*}
\norm{D_t R^2}_0&\leq C\norm{D_t w_c}_0(\norm{w_o}_0+\norm{w_c}_0) + C\|D_t w_o\|_0 \|w_c\|_0
\leq C \delta_{q+1} \delta_q^{\sfrac{1}{2}}\lambda_q\, .
\end{align*}

\medskip

{\bf Estimates on $R^3$ and $D_t R^3$.} The estimates on $\|R^3\|_0$ and $\|R^3\|_1$ are a direct consequence of the mollification estimates \eqref{e:v-v_ell} and \eqref{e:v_ell} as well as Lemma \ref{l:ugly_lemma}.
Moreover, 
\begin{align}
\|D_t R^3\|_0 &\leq \|v-v_\ell\|_0 \|D_t w\|_0 + (\|D_t v\|_0 + \|D_t v_\ell\|) \|w\|_0\nonumber\\
&=\|v-v_\ell\|_0 \left(\|D_t w_c\|_0 + \|D_t w_o\|\right) + (\|D_t v\|_0 + \|D_t v_\ell\|) \|w\|_0
\end{align}
Concerning $D_tv$, note that, by our inductive hypothesis
\begin{align*}
\|D_t v\|_0 &\leq \|\partial_t v + v \cdot \nabla v\|_0 + 
\|v-v_\ell\|_0 \|v\|_1\\ 
&\leq \|p_q\|_1 + \|\mathring{R}_q\|_1 + C\delta_q \lambda_q^2 \ell
\leq C \delta_q\lambda_q\, .
\end{align*}
Thus the required estimate on $D_t R^3$ follows from Lemma \ref{l:ugly_lemma_2}.

\medskip

{\bf Estimates on $R^4$ and $D_t R^4$.}  From the mollification estimates \eqref{e:R-R_ell} and \eqref{e:R_ell} we deduce
\begin{align*}
\|R^4\|_0 &\leq C\|\mathring{R}\|_1 \ell \leq C \delta_{q+1}\lambda_q \ell\\
\|R^4\|_1 &\leq 2 \|\mathring{R}\|_1 \leq C \delta_{q+1} \lambda_q\,.
\end{align*}
As for $D_t R^4$, observe first that, using our inductive hypothesis, 
\begin{align*}
\|D_t \mathring{R}\|_0 &\leq \|\partial_t \mathring{R} + v \cdot \nabla \mathring{R}\|_0 
+ \|v_\ell -v\|_0 \|\mathring{R}\|_1
\leq C\delta_{q+1} \delta_q^{\sfrac{1}{2}} \lambda_q + C \delta_{q+1} \delta_q^{\sfrac{1}{2}} \lambda_q^2 \ell
\end{align*}
Moreover,
\begin{align}
 & D_t \mathring{R}_\ell = (D_t \mathring{R})*\psi_\ell + v_\ell \cdot \nabla \mathring{R}_\ell - 
(v_{\ell} \cdot \nabla \mathring{R})*\psi_\ell\nonumber\\
= &(D_t \mathring{R})*\psi_\ell+\div\left(v_\ell\otimes \mathring{R}_\ell - 
(v\otimes \mathring{R})*\psi_\ell\right) {+ [(v - v_\ell) \cdot \nabla \mathring{R}]*\psi_\ell}\, ,
\label{e:DtRell}
\end{align}
where we have used that $\div v=0$.
Using Proposition \ref{p:CET} we deduce
\begin{align}
&\|v_\ell \otimes \mathring{R}_\ell - (v \otimes \mathring{R})*\psi_\ell\|_1
\leq C  \delta_{q+1} \delta_q^{\sfrac{1}{2}} \lambda_q \lambda_q\ell\, .\label{e:use_CET}
\end{align}
Gathering all the estimates we then achieve 
\[
\|D_t R^4\|_0 \leq \|D_t \mathring{R}\|_0 + \|D_t \mathring{R}_\ell\|_0 \leq C \delta_{q+1} \lambda_q \delta_q^{\sfrac{1}{2}} \, .
\]  
The estimate \eqref{e:R4} follows now using \eqref{e:ordering_params}. 

\medskip

{\bf Estimates on $R^5$.} Recall that $D_t\mathring{R}_{\ell,l}=0$. 
Therefore, using the arguments from \eqref{e:DtRell} 
\begin{align*}
\|D_t(\mathring{R}_\ell-\mathring{R}_{\ell,l})\|_0&  = \|D_t\mathring{R}_\ell\|_0
\leq C\delta_{q+1}\delta_q^{\sfrac12}\lambda_q\, .
\end{align*}
On the other hand, using again the identity \eqref{e:DtRell} and Proposition \ref{p:CET}
\begin{align*}
\|D_t(\mathring{R}_\ell-\mathring{R}_{\ell,l})\|_1& = \|D_t\mathring{R}_\ell\|_1
\leq  C \ell^{-1} \|D_t \mathring{R}\|_0 + \|v_\ell \otimes \mathring{R}_\ell 
- (v\otimes \mathring{R})*\psi_\ell \|_2\\
&+ C \ell^{-1} \|v-v_\ell\|_0\|\mathring{R}\|_1 \leq C \delta_{q+1} \delta_q^{\sfrac12} \lambda_q \ell^{-1}\, . 
\end{align*}
Since $\mathring{R}_{\ell,l}(x,t\mu^{-1})=\mathring{R}_\ell(x,  t\mu^{-1})$,
the difference $\mathring{R}_\ell-\mathring{R}_{\ell,l}$ vanishes at $t=l\mu^{-1}$.
From Proposition \ref{p:transport_derivatives} we deduce that, 
for times $t$ in the support of $\chi_l$ (i.e. $|t-l\mu^{-1}|<\mu^{-1}$), 
\begin{align*}
\|\mathring{R}_\ell-\mathring{R}_{\ell,l}\|_0 &\leq C\mu^{-1} \|D_t(\mathring{R}_\ell-\mathring{R}_{\ell,l})\|_0 \leq C \mu^{-1} \delta_{q+1} \delta_q^{\sfrac{1}{2}} \lambda_q\\
\|\mathring{R}_\ell-\mathring{R}_{\ell,l}\|_1& \leq C\mu^{-1}\|D_t(\mathring{R}_\ell-\mathring{R}_{\ell,l})\|_1 \leq C \mu^{-1} \delta_{q+1} \delta_q^{\sfrac{1}{2}} \lambda_q \ell^{-1}\, .
\end{align*}
The desired estimates on $\|R^5\|_0$ and $\|R^5\|_1$ follow then easily using \eqref{e:ordering_params}.

\medskip

{\bf Estimate on $D_t R^5$.} In this case we compute
\[
D_t R^5 = \sum_l 2 \chi_l \partial_t \chi_l (\mathring{R}_\ell - \mathring{R}_{\ell, l}) + \sum_l \chi_l^2 D_t \mathring{R}_\ell\,.
\]
The second summand has been estimate above and,
since $\|\partial_t \chi_l\|_0 \leq C \mu$, the first summand can be estimated by
$C \mu\delta_{q+1} \delta_q^{\sfrac{1}{2}} \lambda_q \mu^{-1}$ (again appealing to the arguments
above). 

\medskip

{\bf Proof of \eqref{e:Dt_R_all}.} To achieve this last inequality, observe that
\[
\|\partial_t \mathring{R}_1 + v_1\cdot \nabla \mathring{R}_1\|_0 \leq \|D_t \mathring{R}_1\|_0 + \left(\|v-v_\ell\|_0 + \|w\|_0\right)\|\mathring{R}_1\|_1 \, .
\]
On the other hand, by \eqref{e:v-v_ell}, $\|v-v_\ell\|_0 \leq C \delta_q^{\sfrac{1}{2}} \lambda_q \ell$. Moreover, by \eqref{e:W_est_0}, \eqref{e:corrector_est} and \eqref{e:ordering_params} 
$\|w\|\leq \|w_o\|_0 + \|w_c\|_0 \leq  C \delta_{q+1}^{\sfrac{1}{2}}$. Thus, by \eqref{e:allR} we conclude
\begin{align*}
& \|\partial_t \mathring{R}_1 + v_1\cdot \nabla \mathring{R}_1\|_0 \leq
C \left(\mu + \delta_{q+1}^{\sfrac{1}{2}}\lambda_{q+1}\right)\\
&\qquad 
\left(\frac{\delta_{q+1}^{\sfrac{1}{2}} \mu}{\lambda_{q+1}^{1-\eps}} + \frac{\delta_{q+1} \delta_q^{\sfrac{1}{2}} \lambda_q \lambda_{q+1}^\eps}{\mu} +
\delta_{q+1}^{\sfrac{1}{2}} \delta_q^{\sfrac{1}{2}} \lambda_q \ell+\frac{\delta_{q+1}^{\sfrac12}\delta_q\lambda_q}{\lambda_{q+1}^{1-\eps}\mu\ell}\right)
\end{align*}
Since by \eqref{e:ordering_params} $\mu \leq \delta_{q+1}^{\sfrac{1}{2}}\lambda_{q+1}$, \eqref{e:Dt_R_all}
follows easily.
\end{proof}


\section{Conclusion of the proof}\label{s:conclusion}

In Sections \ref{s:perturbations}-\ref{s:reynolds} we showed the construction for a single step, referring to $(v_q,p_q,\mathring{R}_q)$ as 
$(v,p,\mathring{R})$ and to $(v_{q+1},p_{q+1},\mathring{R}_{q+1})$ as $(v_1,p_1,\mathring{R}_1)$. From now on we will consider the full iteration again, hence using again the indices $q$ and $q+1$. 

In order to proceed, recall that the sequences $\{\delta_q\}_{q\in\N}$ and $\{\lambda_q\}_{q\in\N}$ 
are chosen to satisfy
\[
\delta_q=a^{-b^q},\quad a^{cb^{q+1}}\leq \lambda_q \leq 2 a^{cb^{q+1}}
\]
for some given constants $c>5/2$ and $b>1$ and for $a>1$. 
Note that this has the consequence that if $a$ is chosen sufficiently large (depending only on $b>1$) then 
\begin{align}\label{e:condition_deltalambda}
\delta_q^{\sfrac12}\lambda_q^{\sfrac15}\leq 
\delta_{q+1}^{\sfrac12}\lambda_{q+1}^{\sfrac15},\quad\delta_{q+1}\leq\delta_q,\quad\textrm{ and }\quad\lambda_q\leq \lambda_{q+1}^{\frac{2}{b+1}}\,.
\end{align}
\medskip

\subsection{Choice of the parameters $\mu$ and $\ell$}

We start by specifying the parameters $\mu=\mu_q$ and $\ell=\ell_q$: we determine them
optimizing the right hand side of
\eqref{e:allR}. More precisely, we set 
\begin{equation}\label{e:choice_mu}
\mu := \delta_{q+1}^{\sfrac{1}{4}} \delta_q^{\sfrac{1}{4}} \lambda_q^{\sfrac{1}{2}} \lambda_{q+1}^{\sfrac{1}{2}}
\end{equation}
so that the first two expressions in \eqref{e:allR} are equal, and then, having determined $\mu$, set
\begin{equation}\label{e:choice_ell}
\ell := \delta_{q+1}^{-\sfrac{1}{8}}\delta_{q}^{\sfrac{1}{8}}\lambda_{q}^{-\sfrac{1}{4}}\lambda_{q+1}^{-\sfrac{3}{4}}
\end{equation}
so that the last two expressions in \eqref{e:allR} are equal (up to a factor $\lambda_{q+1}^\eps$). 

In turn, these choices lead to 
\begin{align}
\|\mathring{R}_{q+1}\|_0+\frac{1}{\lambda_{q+1}}\|\mathring{R}_{q+1}\|_1
&\leq C\delta_{q+1}^{\sfrac34}\delta_q^{\sfrac14}\lambda_q^{\sfrac12}\lambda_{q+1}^{\eps-\sfrac12}+
C\delta_{q+1}^{\sfrac38}\delta_q^{\sfrac58}\lambda_q^{\sfrac34}\lambda_{q+1}^{\eps-\sfrac34}\nonumber\\
&=C\delta_{q+1}^{\sfrac34}\delta_q^{\sfrac14}\lambda_q^{\sfrac12}\lambda_{q+1}^{\eps-\sfrac12}
\left(1+\left(\frac{\delta_q^{\sfrac12}\lambda_q^{\sfrac13}}{\delta_{q+1}^{\sfrac12}\lambda_{q+1}^{\sfrac13}}\right)^{\sfrac34}\right)\nonumber\\
&\stackrel{\eqref{e:condition_deltalambda}}{\leq} C\delta_{q+1}^{\sfrac34}\delta_q^{\sfrac14}\lambda_q^{\sfrac12}\lambda_{q+1}^{\eps-\sfrac12}\, .\label{e:Rq+1}
\end{align}
Observe also that by \eqref{e:Dt_R_all}, we have
\begin{equation}\label{e:Rq+1_adv}
\|\partial_t \mathring{R}_{q+1} + v_{q+1} \cdot \nabla \mathring{R}_{q+1}\|_0
\leq C \delta_{q+1}^{\sfrac{1}{2}} \lambda_{q+1} \left(\delta_{q+1}^{\sfrac34}\delta_q^{\sfrac14}\lambda_q^{\sfrac12}\lambda_{q+1}^{\eps-\sfrac12}\right).
\end{equation}

Let us check that the conditions \eqref{e:conditions_lambdamu_2} are satisfied for some $\beta>0$ (remember that $\beta$ should be independent of $q$).
To this end we calculate
\begin{align*}
\frac{\delta_q^{\sfrac12}\lambda_q\ell}{\delta_{q+1}^{\sfrac12}}&=\left(\frac{\delta_q^{\sfrac12}\lambda_q^{\sfrac35}}{\delta_{q+1}^{\sfrac12}\lambda_{q+1}^{\sfrac35}}\right)^{\sfrac54}\,,\qquad
\frac{\delta_q^{\sfrac12}\lambda_q}{\mu}=\left(\frac{\delta_q^{\sfrac12}\lambda_q}{\delta_{q+1}^{\sfrac12}\lambda_{q+1}}\right)^{\sfrac12}\,,\\
\frac{1}{\ell\lambda_{q+1}}&=\left(\frac{{\delta_{q+1}^{\sfrac12}}\lambda_q}
{{ \delta_{q}^{\sfrac12}}\lambda_{q+1}}\right)^{\sfrac14}\,,\qquad 
\frac{\mu}{\delta_{q+1}^{\sfrac12}\lambda_{q+1}}=\left(\frac{\delta_{q}^{\sfrac12}\lambda_q}{\delta_{q+1}^{\sfrac12}\lambda_{q+1}}\right)^{\sfrac12}\,.
\end{align*}
Hence the conditions \eqref{e:conditions_lambdamu_2} follow from
\eqref{e:condition_deltalambda} choosing $\beta=\frac{b-1}{5b+5}$.

\subsection{Proof of Proposition \ref{p:iterate}}
Fix the constants $c> \frac{5}{2}$ and $b>1$ 
and also an $\eps>0$ whose choice, like that of $a>1$, will be specified later. 
The proposition is proved inductively. The initial triple is defined to be $(v_0, p_0, \mathring{R}_0)= (0,0,0)$. 
Given now $(v_q, p_q, \mathring{R}_q)$ satisfying the estimates \eqref{e:v_C0_iter}-\eqref{e:R_Dt_iter},
we claim that the triple $(v_{q+1}, p_{q+1}, \mathring{R}_{q+1})$ constructed above  
satisfies again all the corresponding estimates.

\smallskip

\noindent{\bf Estimates on $\mathring{R}_{q+1}$.}
Note first of all that, using the form of the estimates in \eqref{e:allR} and \eqref{e:Dt_R_all}, the estimates \eqref{e:R_C1_iter} and \eqref{e:R_Dt_iter} follow from \eqref{e:R_C0_iter}. On the other hand, in light of \eqref{e:Rq+1}, \eqref{e:R_C0_iter} follows from the recursion relation
$$
C\delta_{q+1}^{\sfrac34}\delta_q^{\sfrac14}\lambda_q^{\sfrac12}\lambda_{q+1}^{\eps-\sfrac12}\leq \eta\delta_{q+2}.
$$
Using our choice of $\delta_q$ and $\lambda_q$ from Proposition \ref{p:iterate}, we see that this inequality is equivalent to
$$
C\leq a^{\frac{1}{4}b^q(1+3b-2cb+(2c-4-4\eps c)b^2)},
$$
which, since $b>1$, is satisfied for all $q\geq 1$ for a sufficiently large fixed constant $a>1$, provided 
$$
\left(1+3b-2cb+(2c-4-4\eps c)b^2\right)>0.
$$
Factorizing, we obtain the inequality $(b-1)((2c-4)b-1)-4\eps cb^2>0$. It is then easy to see that for any $b>1$ and $c>5/2$ there exists
$\eps>0$ so that this inequality is satisfied. In this way we can choose $\eps>0$ (and $\beta$ above) depending solely on $b$ and $c$. Consequently, this choice will determine all the constants in the estimates in Sections \ref{s:perturbations}-\ref{s:reynolds}. We can then pick $a>1$ sufficiently large so that \eqref{e:R_C0_iter}, and hence also \eqref{e:R_C1_iter} and \eqref{e:R_Dt_iter}  hold for $\mathring{R}_{q+1}$.

\smallskip

\noindent{\bf Estimates on $v_{q+1}-v_q$.} By \eqref{e:W_est_0}, Lemma \ref{l:ugly_lemma} and
\eqref{e:conditions_lambdamu_2} we conclude
\begin{align}
\|v_{q+1} - v_q\|_0 &\leq \|w_o\|_0 + \|w_c\|_0 \leq \delta_{q+1}^{\sfrac{1}{2}} \left(\frac{M}{2} +
\lambda_{q+1}^{-\beta}\right)\\
\|v_{q+1} - v_q\|_1 &\leq \|w_o\|_1 + \|w_c\|_1 \leq \delta_{q+1}^{\sfrac{1}{2}} \left(\frac{M}{2} +
\lambda_{q+1}^{-\beta}\right)\, .
\end{align}
Since $\lambda_{q+1} \geq \lambda_1 \geq a^{cb^2}$, for $a$ sufficiently large we conclude
\eqref{e:v_C0_iter} and \eqref{e:v_C1_iter}.

\smallskip

\noindent{\bf Estimate on the energy.} Recall Lemma \ref{l:energy}
and observe that, by \eqref{e:conditions_lambdamu_2}, 
$\frac{\delta_{q+1}^{\sfrac12}\delta_q^{\sfrac12}\lambda_q}{\lambda_{q+1}}\leq 
\frac{\delta_{q+1}^{\sfrac12}\mu}{\lambda_{q+1}}$. Moreover,
\[
\delta_{q+1} \delta_q^{\sfrac{1}{2}} \lambda_q = 
a^{-b^{q+1} - b^q/2 + c b^{q+1}} = a^{b^q ((c-1) b - 1/2)}
\geq a\geq 1\, .
\]
So the right hand side of \eqref{e:energy} is smaller than 
$C \frac{\delta_{q+1} \delta_q^{\sfrac12} \lambda_q}{\mu}
+ C \frac{\delta_{q+1}^{\sfrac12} \mu}{\lambda_{q+1}}$, i.e. smaller (up to a constant factor) than
the right hand side of \eqref{e:allR}. Thus, the argument used above to prove \eqref{e:R_C0_iter}
gives also \eqref{e:energy_iter}.

\smallskip

\noindent{\bf Estimates on $p_{q+1} - p_1$.}
From the definition of $p_{q+1}$ in \eqref{e:def_p_1} we deduce
\begin{align*}
\|p_{q+1}-p_q\|_0 &\leq \frac{1}{2} (\|w_o\|_0 + \|w_c\|_0)^2 + C\ell \|v_q\|_1 \|w\|_0
\end{align*}
As already argued in the estimate for \eqref{e:v_C0_iter}, $\|w_o\| + \|w_c\|\leq M \delta^{\sfrac12}_q$.
Moreover $C \ell \|v_1\|_1\|w\|_0 \leq C M \delta_{q+1}^{\sfrac12} \delta_q^{\sfrac12} \lambda_q \ell$,
which is smaller than the right hand side of \eqref{e:allR}. Having already argued that such
quantity is smaller than $\eta \delta_{q+2}$ we can obviously bound $C \ell \|v_q\|_1 \|w\|_0$
with $\frac{M^2}{2} \delta_{q+1}$. This shows \eqref{e:p_C0_iter}. Moreover,
differentiating \eqref{e:def_p_1} we achieve the
bound
\begin{align*}
\|p_{q+1}-p_q\|_1 &\leq (\|w_o\|_1 + \|w_c\|_1) (\|w_o\|_0 + \|w_c\|_0) + C \delta_{q+1}^{\sfrac{1}{2}}
\delta_q^{\sfrac{1}{2}} \lambda_q \lambda_{q+1} \ell 
\end{align*}
and arguing as above we conclude \eqref{e:p_C1_iter}.

\smallskip

\noindent{\bf Estimates \eqref{e:t_derivatives}.} Here we can use the obvious identity 
$\partial_tw_q=D_tw_q - (v_q)_\ell\cdot\nabla w_q$ together with Lemmas \ref{l:ugly_lemma} and \ref{l:ugly_lemma_2}
to obtain $\|\partial_t v_{q+1} - \partial_t v_q\|_0 \leq C \delta_{q+1}^{\sfrac{1}{2}} \lambda_{q+1}$
Then, using \eqref{e:summabilities}, we conclude $\|\partial_t v_q\|_0 \leq C \delta_q^{\sfrac{1}{2}} \lambda_q$. 

To handle $\partial_t p_{q+1} - \partial_t p_q$ observe first that, by our construction,
\begin{align}
&\|\partial_t (p_{q+1} - p_q)\|_0 \leq (\|w_c \|_0 + \|w_o\|_0) (\|\partial_t w_c\|_0+ \|\partial_t w_o\|_0)\nonumber\\
&\qquad\qquad\qquad\qquad + 2 \|w\|_0\|\partial_t v_q\|_0 
+ \ell\|v_q\|_1\|\partial_t w\|_0\, .\nonumber
\end{align}
As above, we can derive the estimates 
$\|\partial_t w_o\|_0 +\|\partial_t w_c\|_0\leq C \delta_{q+1}^{\sfrac{1}{2}} \lambda_{q+1}$ from Lemmas \ref{l:ugly_lemma} and \ref{l:ugly_lemma_2}.
Hence
\begin{align}
\|\partial_t (p_{q+1} - p_q)\|_0 &\leq C \delta_{q+1} \lambda_{q+1} + C \delta_{q+1}^{\sfrac{1}{2}} \delta_q^{\sfrac{1}{2}} \lambda_q + C \delta_q^{\sfrac{1}{2}} \lambda_q \ell \delta_{q+1}^{\sfrac{1}{2}} \lambda_{q+1}\, .
\end{align}
Since $\ell \leq \lambda_q^{-1}$ and $\delta_q^{\sfrac{1}{2}}\lambda_q\leq \delta_{q+1}^{\sfrac{1}{2}} \lambda_{q+1}$, the desired inequality follows.
This concludes the proof.

\appendix
\section{H\"older spaces}\label{s:hoelder}

In the following $m=0,1,2,\dots$, $\alpha\in (0,1)$, and $\beta$ is a multi-index. We introduce the usual (spatial) 
H\"older norms as follows.
First of all, the supremum norm is denoted by $\|f\|_0:=\sup_{\T^3\times [0,1]}|f|$. We define the H\"older seminorms 
as
\begin{equation*}
\begin{split}
[f]_{m}&=\max_{|\beta|=m}\|D^{\beta}f\|_0\, ,\\
[f]_{m+\alpha} &= \max_{|\beta|=m}\sup_{x\neq y, t}\frac{|D^{\beta}f(x, t)-D^{\beta}f(y, t)|}{|x-y|^{\alpha}}\, ,
\end{split}
\end{equation*}
where $D^\beta$ are {\em space derivatives} only.
The H\"older norms are then given by
\begin{eqnarray*}
\|f\|_{m}&=&\sum_{j=0}^m[f]_j\\
\|f\|_{m+\alpha}&=&\|f\|_m+[f]_{m+\alpha}.
\end{eqnarray*}
Moreover, we will write $[f (t)]_\alpha$ and $\|f (t)\|_\alpha$ when the time $t$ is fixed and the
norms are computed for the restriction of $f$ to the $t$-time slice.

Recall the following elementary inequalities:
\begin{equation}\label{e:Holderinterpolation}
[f]_{s}\leq C\bigl(\varepsilon^{r-s}[f]_{r}+\varepsilon^{-s}\|f\|_0\bigr)
\end{equation}
for $r\geq s\geq 0$, $\eps>0$, and 
\begin{equation}\label{e:Holderproduct}
[fg]_{r}\leq C\bigl([f]_r\|g\|_0+\|f\|_0[g]_r\bigr)
\end{equation}
for any $1\geq r\geq 0$. From \eqref{e:Holderinterpolation} with $\eps=\|f\|_0^{\frac1r}[f]_r^{-\frac1r}$ we obtain the 
standard interpolation inequalities
\begin{equation}\label{e:Holderinterpolation2}
[f]_{s}\leq C\|f\|_0^{1-\frac{s}{r}}[f]_{r}^{\frac{s}{r}}.
\end{equation}

Next we collect two classical estimates on the H\"older norms of compositions. These are also standard, for instance
in applications of the Nash-Moser iteration technique.

\begin{proposition}\label{p:chain}
Let $\Psi: \Omega \to \mathbb R$ and $u: \R^n \to \Omega$ be two smooth functions, with $\Omega\subset \R^N$. 
Then, for every $m\in \mathbb N \setminus \{0\}$ there is a constant $C$ (depending only on $m$,
$N$ and $n$) such that
\begin{align}
\left[\Psi\circ u\right]_m &\leq C ([\Psi]_1 [u]_m+[\Psi]_m \|u\|_0^{m-1} [u]_m)\label{e:chain0}\\
\left[\Psi\circ u\right]_m &\leq C ([\Psi]_1 [u]_m+[\Psi]_m [u]_1^{m} )\, .
\label{e:chain1}
\end{align} 
\end{proposition}

\section{Estimates for transport equations}\label{s:transport_equation}

In this section we recall some well known results regarding smooth solutions of
the \emph{transport equation}:
\begin{equation}\label{e:transport}
\left\{\begin{array}{l}
\partial_t f + v\cdot  \nabla f =g,\\ 
f|_{t_0}=f_0,
\end{array}\right.
\end{equation}
where $v=v(t,x)$ is a given smooth vector field. 
We denote the material derivative $\partial_t+v\cdot \nabla$ by $D_t$. We will consider solutions
on the entire space $\R^3$ and treat solutions on the torus simply as periodic solution in $\R^3$.

\begin{proposition}\label{p:transport_derivatives}
Any solution $f$ of \eqref{e:transport} satisfies
\begin{align}
\|f (t)\|_0 &\leq \|f_0\|_0 + \int_{t_0}^t \|g (\tau)\|_0\, d\tau\,,\label{e:max_prin}\\
[f(t)]_1 &\leq [f_0]_1e^{(t-t_0)[v]_1} + \int_{t_0}^t e^{(t-\tau)[v]_1} [g (\tau)]_1\, d\tau\,,\label{e:trans_est_0}
\end{align}
and, more generally, for any $N\geq 2$ there exists a constant $C=C_N$ so that
\begin{align}
[f (t)]_N & \leq \Bigl([ f_0]_N + C(t-t_0)[v]_N[f_0]_1\Bigr)e^{C(t-t_0)[v]_1}+\nonumber\\
&\qquad +\int_{t_0}^t e^{C(t-\tau)[v]_1}\Bigl([g (\tau)]_N + [v ]_N [g (\tau)]_{1}\Bigr)\,d\tau.
\label{e:trans_est_1}
\end{align}
Define $\Phi (t, \cdot)$ to be the inverse of the flux $X$ of $v$ starting at time $t_0$ as the identity
(i.e. $\frac{d}{dt} X = v (X,t)$ and $X (x, t_0 )=x$). Under the same assumptions as above:
\begin{align}
\norm{D\Phi (t) -\Id}_0&\leq e^{(t-t_0)[v]_1}-1\,,  \label{e:Dphi_near_id}\\
[\Phi (t)]_N &\leq C(t-t_0)[v]_Ne^{C(t-t_0)[v]_1} \qquad \forall N \geq 2.\label{e:Dphi_N}
\end{align}
\end{proposition}
\begin{proof} 
We start with the following elementary observation for transport equations: if $f$ solves \eqref{e:transport}, then 
$\frac{d}{dt}f(X(t,x),t)=g(X(t,x), t)$ and consequently
$$
f(t,x)=f_0(\Phi(x,t))+\int_{t_0}^tg(X(\Phi(t,x), \tau), \tau)\,d\tau.
$$
The maximum principle \eqref{e:max_prin} follows immediately. 
Next, differentiate \eqref{e:transport} in $x$ to obtain the identity
$$
D_tDf=(\partial_t+v\cdot\nabla)Df=Dg-DfDv.
$$
Applying \eqref{e:max_prin} to $Df$ yields
$$
[f(t)]_1\leq [f_0]_1+\int_{t_0}^t\left([g(\tau)]_1+[v]_1[f(\tau)]_1\right)\,d\tau.
$$
An application of Gronwall's inequality then yields \eqref{e:trans_est_0}.

More generally, differentiating \eqref{e:transport} $N$ times yields
\begin{equation}\label{e:diff_transport}
\partial_t  D^N f  + (v\cdot \nabla) D^Nf = D^N g + \sum_{j=0}^{N-1} c_{j,N}D^{j+1} f : D^{N-j} v
\end{equation}
(where $:$ is a shorthand notation for sums of products of entries of the corresponding tensors). 

Also, using \eqref{e:max_prin} and the interpolation inequality \eqref{e:Holderinterpolation2} we can estimate
$$
[f(t)]_N\leq [f_0]_N+\int_{t_0}^t\left([g(\tau)]_N+C\bigl([v]_N[f(\tau)]_1+[v]_1[f(\tau)]_N\bigr)\right)\,d\tau.
$$
Plugging now the estimate \eqref{e:trans_est_0}, Gronwall's inequality leads -- after some elementary calculations -- to \eqref{e:trans_est_1}.

\smallskip

The estimate \eqref{e:Dphi_N} follows easily from \eqref{e:trans_est_1} observing that
$\Phi$ solves \eqref{e:transport} with $g=0$ and $D^2 \Phi (\cdot, t_0) = 0$. 
Consider next $\Psi (x,t) = \Phi (x,t) -x$ and 
observe first that $\partial_t \Psi + v\cdot \nabla \Psi = -v$. Since $D \Psi (\cdot, t_0)=0$,
we apply \eqref{e:trans_est_0} to conclude 
\[
[\Psi (t)]_1 \leq \int_{t_0}^t e^{(t-\tau)[v]_1} [v]_1 d\tau = e^{(t-t_0) [v]_1} -1\, . 
\]
Since $D\Psi (x,t) = D\Phi (x,t) - {\rm Id}$, \eqref{e:Dphi_near_id} follows.
\end{proof}

\section{Constantin-E-Titi commutator estimate}\label{s:CET}

Finally, we recall the quadratic commutator estimate from \cite{ConstantinETiti} 
(cf. also with \cite[Lemma 1]{CDSz}):

\begin{proposition}\label{p:CET}
Let $f,g\in C^{\infty}(\T^3\times\T)$ and $\psi$ the mollifier of Section \ref{s:perturbations}. For any $r\geq 0$ we have the estimate
\[
\Bigl\|(f*\chi_\ell)( g*\chi_\ell)-(fg)*\chi_\ell\Bigr\|_r\leq C\ell^{2-r}\|f\|_1\|g\|_1\, ,
\]
where the constant $C$ depends only on $r$.
\end{proposition}

\section{Schauder Estimates}
\label{s:schauder}

We recall here the following consequences of the classical Schauder estimates  (cf.\ 
\cite[Proposition 5.1]{DS3}).

\begin{proposition}\label{p:GT}
For any $\alpha\in (0,1)$ and any $m\in \N$ there exists a constant $C (\alpha, m)$ with the following properties.
If $\phi, \psi: \T^3\to \R$ are the unique solutions of
\[
\left\{\begin{array}{l}
\Delta \phi = f\\ \\
\fint \phi =0
\end{array}\right.
\qquad\qquad 
\left\{\begin{array}{l}
\Delta \psi = {\rm div}\, F\\ \\
\fint \psi =0
\end{array}\right. \, ,
\]
then
\begin{equation}\label{e:GT_laplace}
\|\phi\|_{m+2+\alpha} \leq C (m, \alpha) \|f\|_{m, \alpha}
\quad\mbox{and}\quad \|\psi\|_{m+1+\alpha} \leq C (m, \alpha) \|F\|_{m, \alpha}\, .
\end{equation}
Moreover we have the estimates
\begin{eqnarray}
&&\|\mathcal{R} v\|_{m+1+\alpha} \leq C (m,\alpha) \|v\|_{m+\alpha}\label{e:Schauder_R}\\
&&\|\mathcal{R} ({\rm div}\, A)\|_{m+\alpha}\leq C(m,\alpha) \|A\|_{m+\alpha}\label{e:Schauder_Rdiv}
\end{eqnarray}
\end{proposition}

\section{Stationary phase lemma}\label{s:stationary}

We recall here the following simple facts (for a proof we refer to \cite{DS3})

\begin{proposition}\label{p:stat_phase}
(i) Let $k\in\Z^3\setminus\{0\}$ and $\lambda\geq 1$ be fixed. 
For any $a\in C^{\infty}(\T^3)$ and $m\in\N$ we have
\begin{equation}\label{e:average}
\left|\int_{\T^3}a(x)e^{i\lambda k\cdot x}\,dx\right|\leq \frac{[a]_m}{\lambda^m}.
\end{equation}

(ii) Let $k\in\Z^3\setminus\{0\}$ be fixed. For a smooth vector field $a\in C^{\infty}(\T^3;\R^3)$ let 
$F(x):=a(x)e^{i\lambda k\cdot x}$. Then we have
$$
\|\RR(F)\|_{\alpha}\leq \frac{C}{\lambda^{1-\alpha}}\|a\|_0+\frac{C}{\lambda^{m-\alpha}}[a]_m+\frac{C}{\lambda^m}[a]_{m+\alpha},
$$
where $C=C(\alpha,m)$.
\end{proposition}

\section{One further commutator estimate}\label{s:commutator}

\begin{proposition}\label{p:commutator}
Let $k\in\Z^3\setminus\{0\}$ be fixed. For any smooth vector field $a\in C^\infty  (\T^3;\R^3)$ and any smooth function $b$, if we set $F(x):=a(x)e^{i\lambda k\cdot x}$, we then have
\begin{align}
&\|[b, \mathcal{R}] (F)\|_\alpha \leq  \frac{C}{\lambda^{2-\alpha}}\|a\|_0 \|b\|_1
+\frac{C}{\lambda^{m-\alpha}}(\|a\|_m\|b\|_2 + \|a\|_2\|b\|_m)\nonumber\\ 
&+\frac{C}{\lambda^{m}}(\|a\|_{m+\alpha} \|b\|_{2+\alpha} + \|a\|_{2+\alpha} \|b\|_{m+\alpha}),
\end{align}
where $C=C(\alpha,m)$.
\end{proposition}

\begin{proof} {\bf Step 1} First of all, given a vector field $v$ define the operator 
\[
\mathcal{S} (v) := \nabla v + (\nabla v)^t - \frac{2}{3} (\div v) \Id\, .
\]
First observe that
\begin{equation}\label{e:kernel}
\div \mathcal{S} (v) = 0 \qquad \iff \qquad v \equiv \mbox{ const.}
\end{equation}
One implication is obvious. Next, assume $\div \mathcal{S} (v) = 0$. This is equivalent to the equations
\begin{equation}\label{e:div_S=0}
\Delta v_j+\frac13 \partial_j \div v = 0
\end{equation}
Differentiating and summing in $j$ we then conclude
\[
\frac{4}{3} \Delta \div v = 0\, .
\]
Thus $\div v$ must be constant and, since any divergence has average zero, we conclude that $\div v =0$. Thus
\eqref{e:div_S=0} implies that $\Delta v_i =0$ for every $i$, which in turn gives the desired conclusion.

From this observation we conclude the identity
\begin{equation}\label{e:almost_inverse}
\mathcal{R} (v) = \mathcal{S} (v) + \mathcal{R} (v- \div \mathcal{S} (v))\, \qquad \mbox{for all $v\in C^\infty (\T^3, \R^3)$.}
\end{equation}
Indeed, observe first that $\mathcal{R} (v) = \mathcal{S} (w)$, where $w =\frac14 \mathcal{P} (v) +\frac34 v$. Thus, applying the argument above, since both sides of \eqref{e:almost_inverse} have zero averages, it suffices to show that they have the same divergence. But since $\div \mathcal{R} (v) = v - \fint v$, applying the divergence we obtain
\[
v - \fint v = \div \mathcal{S} (v) + v - \fint v - \div \mathcal{S} (v)\, ,
\]
which is obviously true.

\medskip

{\bf Step 2} Next, for $a\in C^\infty (\T^3, \R^3)$, $k\in \Z^3\setminus \{0\}$ and $\lambda \in \N\setminus \{0\}$, consider 
\[
\mathscr{S} (a e^{i\lambda k\cdot x}) := -\mathcal{S} \left( \frac{3}{4} \frac{a}{\lambda^2 |k|^2} e^{i\lambda k\cdot x} + \frac{1}{4\lambda^2|k|^2} \left(a - \frac{(a\cdot k) k}{|k|^2}\right) e^{i\lambda k\cdot x} \right)\, .
\]
Observe that
\begin{equation}\label{e:operator_A}
\mathscr{S} (b a e^{i\lambda k\cdot x}) - b \mathscr{S} (a e^{i\lambda k\cdot x}) = \frac{a A(b)}{\lambda^2} e^{i\lambda k\cdot x}\, ,
\end{equation}
where $A$ is an homogeneous differential operator of order one with constant coefficients 
(all depending only on $k$). Moreover,
\[
a e^{i\lambda k\cdot x} - \div \mathscr{S}  (a e^{i\lambda k\cdot x}) =
\frac{B_1 (a)}{\lambda} e^{i\lambda k\cdot x} + \frac{B_2 (a)}{\lambda^2} e^{i\lambda k\cdot x}\, ,
\]
where $B_1$ and $B_2$ are homogeneous differential operators of order $1$ and $2$ (respectively) with constant coefficients (again all depending only on $k$).

We use then the identity \eqref{e:almost_inverse} to write
\begin{align}
&-[b, \mathcal{R}] (F) =\RR (bF)- b \RR (F) =  \mathscr{S} (b a e^{i\lambda k\cdot x}) - b \mathscr{S} (a e^{i\lambda k\cdot x})\nonumber\\
&\quad + \mathcal{R} \left( bF - \div \mathscr{S}  (b a e^{i\lambda k\cdot x})\right) - b \mathcal{R} \left( F - \div \mathscr{S}  (a e^{i\lambda k\cdot x})\right)\nonumber\\
= & \frac{a A(b)}{\lambda^2} e^{i\lambda k\cdot x} + \mathcal{R} \left(\frac{B_1 (ab)}{\lambda} e^{i\lambda k\cdot x} + \frac{B_2 (ab)}{\lambda^2} e^{i\lambda k\cdot x}\right)\nonumber\\
&\quad - b \mathcal{R} \left(\frac{B_1 (a)}{\lambda} e^{i\lambda k\cdot x} + \frac{B_2 (a)}{\lambda^2} e^{i\lambda k\cdot x}\right) \, .
\end{align}
Using the Leibniz rule we can write $B_1 (ab) = B_1 (a) b + a B_1 (b)$ and $B_2 (ab) = B_2 (a) b+ a B_2 (b) + C_1 (a) C_1 (b)$, where $C_1$ is an homogeneous operator of order $1$. We can then reorder all terms to write
\begin{align}
-[b, \mathcal{R}] (F) &=  \frac{a A(b)}{\lambda^2} e^{i\lambda k\cdot x}\nonumber\\
&\quad+ \mathcal{R} \left(\frac{a B_1 (b)}{\lambda} e^{i\lambda k\cdot x}\right) + \RR\left(\frac{a B_2 (b) + C_1 (a) C_1 (b)}{\lambda^2} e^{i\lambda k\cdot x}\right)\nonumber\\
&\quad - \frac{1}{\lambda} [b, \RR]  \left( (B_1 (a)) e^{i\lambda k\cdot x}\right) 
- \frac{1}{\lambda^2} [b, \RR]   \left(b (B_2 (a)) e^{i\lambda k\cdot x}\right) \, .
\end{align}
In the first two summands appear only derivatives of $b$, but there are no zero order terms in $b$. We can then estimate the two terms in the second line applying Proposition \ref{p:stat_phase}, with $m=N-1$ to the first summand and with $m=N-2$ to the second summand. Applying in addition interpolation identities, we conclude
\begin{align}
& \|[b, \mathcal{R}] (F)\|_\alpha \leq C\frac{\|a\|_0 \|b\|_1}{\lambda^{2-\alpha}} + C \frac{\|a\|_2 \|b\|_N +\|a\|_{N} \|b\|_2}{\lambda^{N-\alpha}}\nonumber\\
&\quad  + C \frac{\|a\|_{2+\alpha} \|b\|_{N+\alpha} +\|a\|_{N+\alpha} \|b\|_{2+\alpha}}{\lambda^N}\nonumber\\
&\quad + \frac{1}{\lambda} \underbrace{\left\| [b, \RR]  \left(b (B_1 (a)) e^{i\lambda k\cdot x}\right)\right\|_\alpha}_{II}
+ \frac{1}{\lambda^2}  \left\| [b, \RR]  \left(b (B_2 (a)) e^{i\lambda k\cdot x}\right)\right\|_\alpha.\label{e:first_expansion}
\end{align}
(Indeed the above estimate is sub-optimal as we get, for instance, terms of type $\|a\|_0 \|b\|_N + \|a\|_1 \|b\|_{N-1}+ \|a\|_2 \|b\|_{N-2}$ instead of $\|a\|_2 \|b\|_N$; however, this crude estimate is still sufficient for our purposes.)

\medskip

{\bf Step 3} We can now apply the same idea to the term $II$ in \eqref{e:first_expansion} to reach the estimate
\begin{align}
& \|[b, \RR] (F)\|_\alpha \leq C\frac{\|a\|_0 \|b\|_1}{\lambda^{2-\alpha}} + C \frac{\|a\|_2 \|b\|_N +\|a\|_{N} \|b\|_2}{\lambda^{N-\alpha}}\nonumber\\
&\quad  + C \frac{\|a\|_{2+\alpha} \|b\|_{N+\alpha} +\|a\|_{N+\alpha} \|b\|_{2+\alpha}}{\lambda^N} + 
\frac{1}{\lambda^2} \left\| [b, \RR] \left(b (B'_2 (a)) e^{i\lambda k\cdot x}\right)\right\|_\alpha\nonumber\, .
\end{align}
where $B'_2 = B_2 + B_1\circ B_1$ is second order and this time we have applied Proposition \ref{p:stat_phase} with $m=N-2$ and $m=N-3$ to handle the corresponding two terms of order $\lambda^{-2}$ and $\lambda^{-3}$ arising from $II$. Proceeding now inductively, we end up with
\begin{align}
& \|\RR (bF)- b \RR (F)\|_0 \leq \frac{\|a\|_0 \|b\|_1}{\lambda^2} + C \frac{\|a\|_2 \|b\|_N +\|a\|_{N} \|b\|_2}{\lambda^{N-\alpha}}\nonumber\\
&\quad  + C \frac{\|a\|_{2+\alpha} \|b\|_{N+\alpha} +\|a\|_{N+\alpha} \|b\|_{2+\alpha}}{\lambda^N}\nonumber\\
&\quad+ \frac{1}{\lambda^N}  \left\| \mathcal{R}  \left(b (B'_N (a)) e^{i\lambda k\cdot x}\right) -
b \mathcal{R}  \left(B'_N (a) e^{i\lambda k\cdot x}\right)\right\|_\alpha\label{e:N_expansion}\, .
\end{align}
Finally, applying Proposition \ref{p:GT} to the final term we reach the desired estimate.
\end{proof}

\bibliographystyle{acm}
\bibliography{eulerbib}

\begin{thebibliography}{10}

\bibitem{CCFS2007}
{\sc Cheskidov, A., Constantin, P., Friedlander, S., and Shvydkoy, R.}
\newblock Energy conservation and {O}nsager's conjecture for the {E}uler
  equations.
\newblock {\em Nonlinearity 21}, 6 (2008), 1233--1252.

\bibitem{ConstantinETiti}
{\sc Constantin, P., E, W., and Titi, E.~S.}
\newblock Onsager's conjecture on the energy conservation for solutions of
  {E}uler's equation.
\newblock {\em Comm. Math. Phys. 165}, 1 (1994), 207--209.

\bibitem{CDSz}
{\sc Conti, S., De~Lellis, C., and Sz{\'e}kelyhidi, Jr., L.}
\newblock $h$-principle and rigidity for {$C^{1,\alpha}$} isometric embeddings.
\newblock In {\em Nonlinear Partial Differential Equations}, vol.~7 of {\em
  Abel Symposia}. Springer, 2012, pp.~83--116.

\bibitem{MR0213764}
{\sc Courant, R., Friedrichs, K., and Lewy, H.}
\newblock On the partial difference equations of mathematical physics.
\newblock {\em IBM J. Res. Develop. 11\/} (1967), 215--234.

\bibitem{Daneri}
{\sc Daneri, S.}
\newblock Cauchy problem for dissipative {H}\"older solutions to the
  incompressible {E}uler equations.
\newblock {\em Preprint.\/} (2013), 1--33.

\bibitem{DS1}
{\sc De~Lellis, C., and Sz{\'e}kelyhidi, Jr., L.}
\newblock The {E}uler equations as a differential inclusion.
\newblock {\em Ann. of Math. (2) 170}, 3 (2009), 1417--1436.

\bibitem{DS2}
{\sc De~Lellis, C., and Sz{\'e}kelyhidi, Jr., L.}
\newblock On admissibility criteria for weak solutions of the {E}uler
  equations.
\newblock {\em Arch. Ration. Mech. Anal. 195}, 1 (2010), 225--260.

\bibitem{DS4}
{\sc De~Lellis, C., and Sz{\'e}kelyhidi, Jr., L.}
\newblock Dissipative {E}uler flows and {O}nsager's conjecture.
\newblock {\em Preprint.\/} (2012), 1--40.

\bibitem{DSsurvey}
{\sc De~Lellis, C., and Sz{\'e}kelyhidi, Jr., L.}
\newblock The $h$-principle and the equations of fluid dynamics.
\newblock {\em Bull. Amer. Math. Soc. (N.S.) 49}, 3 (2012), 347--375.

\bibitem{DS3}
{\sc De~Lellis, C., and Sz{\'e}kelyhidi, Jr., L.}
\newblock Dissipative continuous {E}uler flows.
\newblock {\em To appear in Inventiones\/} (2013), 1--26.

\bibitem{RobertDuchon}
{\sc Duchon, J., and Robert, R.}
\newblock Inertial energy dissipation for weak solutions of incompressible
  {E}uler and {N}avier-{S}tokes equations.
\newblock {\em Nonlinearity 13}, 1 (2000), 249--255.

\bibitem{Eyink}
{\sc Eyink, G.~L.}
\newblock Energy dissipation without viscosity in ideal hydrodynamics. {I}.
  {F}ourier analysis and local energy transfer.
\newblock {\em Phys. D 78}, 3-4 (1994), 222--240.

\bibitem{EyinkSreenivasan}
{\sc Eyink, G.~L., and Sreenivasan, K.~R.}
\newblock Onsager and the theory of hydrodynamic turbulence.
\newblock {\em Rev. Modern Phys. 78}, 1 (2006), 87--135.

\bibitem{FrischBook}
{\sc Frisch, U.}
\newblock {\em Turbulence}.
\newblock Cambridge University Press, Cambridge, 1995.
\newblock The legacy of A. N. Kolmogorov.

\bibitem{Isett}
{\sc Isett, P.}
\newblock H{\"o}lder continuous {E}uler flows in three dimensions with compact
  support in time.
\newblock {\em Preprint\/} (2012), 1--173.

\bibitem{Kolmogorov}
{\sc Kolmogorov, A.~N.}
\newblock The local structure of turbulence in incompressible viscous fluid for
  very large {R}eynolds numbers.
\newblock {\em Proc. Roy. Soc. London Ser. A 434}, 1890 (1991), 9--13.
\newblock Translated from the Russian by V. Levin, Turbulence and stochastic
  processes: Kolmogorov's ideas 50 years on.

\bibitem{Onsager}
{\sc Onsager, L.}
\newblock Statistical hydrodynamics.
\newblock {\em Nuovo Cimento (9) 6}, Supplemento, 2 (Convegno Internazionale di
  Meccanica Statistica) (1949), 279--287.

\bibitem{Robert}
{\sc Robert, R.}
\newblock Statistical hydrodynamics ({O}nsager revisited).
\newblock In {\em Handbook of mathematical fluid dynamics, {V}ol. {II}}.
  North-Holland, Amsterdam, 2003, pp.~1--54.

\bibitem{Scheffer93}
{\sc Scheffer, V.}
\newblock An inviscid flow with compact support in space-time.
\newblock {\em J. Geom. Anal. 3}, 4 (1993), 343--401.

\bibitem{Shnirelmandecrease}
{\sc Shnirelman, A.}
\newblock Weak solutions with decreasing energy of incompressible {E}uler
  equations.
\newblock {\em Comm. Math. Phys. 210}, 3 (2000), 541--603.

\end{thebibliography}

\end{document}